\documentclass[hidelinks, 10pt]{amsart}

\usepackage[margin=4cm]{geometry}

\usepackage{amssymb}

\usepackage{cite}
\usepackage{hyperref}
\usepackage[all]{xy}
\usepackage{enumerate}
\usepackage{enumitem}
\usepackage{bbm}
\usepackage{tikz-cd}
\usepackage{extpfeil}
\usepackage{relsize}
\usepackage[bbgreekl]{mathbbol}
\usepackage{slashed}
\usepackage{amsfonts}

\newcommand{\eps}{\epsilon}

\newcommand{\bF}{{\mathbb F}}
\newcommand{\bG}{{\mathbb G}}

\newcommand{\bN}{{\mathbb N}}

\newcommand{\bP}{{\mathbb P}}
\newcommand{\bQ}{{\mathbb Q}}

\newcommand{\bW}{{\mathbb W}}

\newcommand{\bZ}{{\mathbb Z}}

\newcommand{\cC}{{\mathcal C}}
\newcommand{\cD}{{\mathcal D}}

\newcommand{\cF}{{\mathcal F}}

\newcommand{\cH}{{\mathcal H}}

\newcommand{\cO}{{\mathcal O}}
\newcommand{\cP}{{\mathcal P}}

\newcommand{\fg}{{\mathfrak g}}

\newcommand{\ta}{\widetilde{a}}

\newcommand{\tw}{\widetilde{w}}

\newcommand{\tF}{\widetilde{F}}

\newcommand{\tS}{\widetilde{S}}

\newcommand{\tX}{\widetilde{X}}

\DeclareMathOperator*{\colim}{colim}

\DeclareMathOperator{\gr}{{gr}}

\DeclareMathOperator{\Spf}{{Spf}}
\DeclareMathOperator{\Sym}{{Sym}}

\DeclareMathOperator{\rk}{{rk}}

\newcommand{\RGamma}{\mathrm{R}\Gamma}

\DeclareMathOperator{\Hom}{{Hom}}

\DeclareMathOperator{\id}{{id}}

\DeclareMathOperator{\op}{{op}}

\DeclareMathOperator{\Spec}{{Spec}}

\DeclareMathOperator{\Map}{Map}
\DeclareMathOperator{\Rep}{Rep}

\newcommand{\dR}{\mathrm{dR}}

\newcommand{\triv}{\mathrm{triv}}
\DeclareMathOperator{\QC}{\mathcal{QC}}

\DeclareMathOperator{\Tot}{Tot}
\DeclareMathOperator{\cone}{cone}
\DeclareMathOperator{\Spin}{Spin}
\DeclareMathOperator{\PGL}{PGL}
\DeclareMathOperator{\PSO}{PSO}
\DeclareMathOperator{\PSp}{PSp}

\newcommand{\Fil}{\mathrm{Fil}}

\newcommand{\an}{\mathrm{an}}

\newcommand{\RHom}{\mathrm{RHom}}
\newcommand{\conj}{\mathrm{conj}}

\newcommand{\obj}{\mathrm{obj}}
\newcommand{\fppf}{\mathrm{fppf}}
\newcommand{\St}{\mathrm{St}}
\newcommand{\CAlg}{\mathrm{CAlg}}
\newcommand{\perf}{\mathrm{perf}}
\newcommand{\cl}{\mathrm{cl}}

\newcommand{\tcF}{\widetilde{\cF}}
\newcommand{\Sm}{\mathrm{Sm}}

\DeclareSymbolFontAlphabet{\mathbb}{AMSb} %to ensure that the meaning of \mathbb does not change
\DeclareSymbolFontAlphabet{\mathbbl}{bbold}

\newcommand{\comment}[1]{}

\makeatletter

\newcommand{\Rmnum}[1]{\expandafter\@slowromancap\romannumeral #1@}
\makeatother

\newtheorem{pr}{Proposition}[section]

\newtheorem{thm}[pr]{Theorem}

\newtheorem{lm}[pr]{Lemma}

\newtheorem{cor}[pr]{Corollary}

\theoremstyle{definition}
\newtheorem{rem}[pr]{Remark}

\theoremstyle{definition}
\newtheorem{example}[pr]{Example}
\newtheorem{definition}[pr]{Definition}
\newtheorem{constr}[pr]{Construction}

\numberwithin{equation}{section}

\begin{document}

\author[]{Alexander Petrov}
\title[Decomposition of de Rham complex for quasi-F-split varieties]{Decomposition of de Rham complex for quasi-F-split varieties}

\begin{abstract}
Using the de Rham stack of Bhatt-Lurie and Drinfeld, we prove that de Rham complex of a smooth quasi-F-split variety over a perfect field of positive characteristic decomposes in all degrees. In particular, smooth proper quasi-F-split varieties have degenerate Hodge-to-de Rham spectral sequence, and satisfy Kodaira-Akizuki-Nakano vanishing. We apply this to prove that the Hodge-to-de Rham spectral sequence for the classifying stack of a reductive group over a field of positive characteristic degenerates.
\end{abstract}

\maketitle

\section{Introduction}

Cohomological invariants of smooth proper varieties over fields of positive characteristic in general fail to satisfy some of the good properties exhibited in characteristic zero, such as degeneration of the Hodge-to-de Rham spectral sequence and Kodaira vanishing. Deligne and Illusie \cite{deligne-illusie} made a fundamental discovery that if a smooth proper variety $X$ over a perfect field $k$ of characteristic $p>0$ lifts over the ring $W_2(k)$ of length $2$ Witt vectors and $\dim X\leq p$, then Hodge-to-de Rham degeneration and Kodaira-Akizuki-Nakano vanishing hold for $X$. They deduced these cohomological statements by proving that the de Rham complex of $X$ decomposes in the derived category of sheaves on $X$. 

For varieties of arbitrary dimension liftability alone is not enough to guarantee degeneration of the Hodge-to-de Rham spectral sequence, as examples in \cite[Theorem 1.1]{petrov} show. In this paper, we point out that for Frobenius-split, and more generally, quasi-Frobenius-split smooth varieties the de Rham complex nevertheless decomposes, regardless of the dimension of $X$.

\begin{thm}[{Corollaries \ref{qfsplit: qfsplit decomposable}, \ref{qfsplit: htdgr akn}}]\label{intro: main}
For a smooth quasi-$F$-split variety $X$ over a perfect field $k$ of characteristic $p$ the de Rham complex $F_{X/k*}\Omega^{\bullet}_{X/k}$ is quasi-isomorphic to the direct sum $\bigoplus\limits_{i\geq 0}\Omega^i_{X^{(1)}/k}[-i]$ of its cohomology sheaves, as a complex of sheaves of $\cO_{X^{(1)}}$-modules.

In particular, for a smooth proper quasi-$F$-split variety $X$ over $k$ the Hodge-to-de Rham spectral sequence $E_1^{i,j}=H^j(X,\Omega^i_{X/k})\Rightarrow H^{i+j}_{\dR}(X/k)$ degenerates at the first page, and the Kodaira-Akizuki-Nakano vanishing holds: for an ample line bundle $L$ on $X$ the cohomology group $H^i(X,\Omega^j_{X/k}\otimes L)$ vanishes when $i+j>\dim X$. 
\end{thm}

An $\bF_p$-scheme is called quasi-$F$-split if for some $n$, the natural map to the mod $p$ reduction of the length $n$ Witt vectors sheaf $\cO_X\xrightarrow{f\mapsto [f]^p}F_*W_n(\cO_X)/p$ has an $\cO_X$-linear splitting. This notion was introduced by Yobuko \cite{yobuko} as a generalization of the notion of $F$-split varieties (corresponding to the case $n=1$) \cite{mehta-ramanathan}, and it is remarkable in that many cohomological consequences of the existence of an $F$-splitting happen to remain valid in the presence of the weaker structure of a quasi-$F$-splitting.

Theorem \ref{intro: main} implies (Corollary \ref{qfsplit: flag varieties}) that for (partial) flag varieties of reductive groups, in particular smooth quadrics, the de Rham complex decomposed, which answers some of the questions raised in \cite[7.11]{illusie-hodge}. When $\dim X\leq p$, Theorem \ref{intro: main} follows readily from \cite[Lemme 2.9]{deligne-illusie} because quasi-$F$-split varieties are liftable over $W_2(k)$ \cite[Theorem 4.4]{yobuko}, \cite[7.2]{kttwyy}. It was also proven for $F$-split varieties of dimension $p+1$ by Achinger-Suh \cite[Theorems A.2, A.4]{achinger-suh} as a consequence of the existence fo the Sen operator on the de Rham complex of a liftable variety, introduced by Bhatt-Lurie and Drinfeld. Kodaira vanishing (i.e. the case $j=\dim X$ in the above) for quasi-$F$-split varieties has also been previously shown in \cite[Theorem 1.7]{nakkajima-yobuko}. Apart from these special cases the above theorem is new, even in the case of $F$-split varieties.

We deduce Theorem \ref{intro: main} from the following decomposition result valid for arbitrary smooth varieties:

\begin{thm}\label{intro: decomposition}
For a smooth variety $X$ over $k$ there is a natural quasi-isomorphism in the derived category of $F_*W(\cO_{X})/p$-modules:
\begin{equation}
F_*W(\cO_{X})/p\otimes_{\cO_{X^{(1)}}} F_{X/k*}\Omega^{\bullet}_{X/k}\simeq \bigoplus\limits_{i\geq 0} F_*W(\cO_{X})/p\otimes_{\cO_{X^{(1)}}}\Omega^i_{X^{(1)}}[-i]
\end{equation}
where the sheaf $F_*W(\cO_{X})/p$ is viewed as an $\cO_{X^{(1)}}$-module via the map of sheaves of algebras $\cO_{X^{(1)}}\to F_*W(\cO_{X})/p$ given by $\cO_{X}\otimes_{k,\varphi_k}k\ni f\otimes 1\mapsto [f^p]$. In particular, there is a natural quasi-isomorphism
\begin{equation}\label{intro: frob decomposition formula}
F_{X/k}^*F_{X/k*}\Omega^{\bullet}_{X/k}\simeq \bigoplus\limits_{i\geq 0}F_{X/k}^*\Omega^i_{X^{(1)}/k}[-i]
\end{equation}
in the derived category of quasi-coherent sheaves on $X$.
\end{thm}

The proof of Theorem \ref{intro: main} accesses the de Rham complex through the de Rham stack, introduced in the setting of varieties in positive characteristic by Drinfeld \cite{drinfeld-prismatization} and Bhatt-Lurie \cite{apc}, \cite{bhatt-lurie-prismatization}. A secondary goal of this paper is to demonstrate the utility of the stacky approach to de Rham cohomology -- the proof crucially uses the gerbe property of the de Rham stack of a smooth variety, and we are unaware of any elementary reformulation of this property in terms of the de Rham complex itself.

The decomposition after Frobenius pullback (\ref{intro: frob decomposition formula}) has been previously shown by Bhargav Bhatt by the same method, and by Vadim Vologodsky by a related method relying on the Azumaya property of the algebra of differential operators on $X$. The aforementioned gerbe property of the de Rham stack can be recovered from this Azumaya property, and Theorem \ref{intro: decomposition} can be proved using it instead of the de Rham stack.

We also extend (Proposition \ref{stacks: frobenius decomposition}) decomposition (\ref{intro: frob decomposition formula}) to the case of smooth Artin stacks, by proving smooth descent for Frobenius pullback of sheaves of differential forms on smooth schemes. It turns out that the classifying stack of any reductive group is Frobenius-split:

\begin{pr}[Proposition \ref{stacks: bg is fsplit}]\label{intro: BG fsplit}
Let $G$ be a reductive group over a perfect field $k$ of characteristic $p>0$. Then the natural map $\cO_{BG^{(1)}}\xrightarrow{F^{\#}_{BG}} RF_{BG*}\cO_{BG}$ in the derived category of quasi-coherent sheaves on the classifying stack $BG$ admits a splitting.
\end{pr}

As a consequence of a version of Theorem \ref{intro: main} for stacks (Thoerem \ref{stacks: frobenius decomposition}), we deduce:

\begin{thm}[{Theorem \ref{stacks: main htdr bg}}]\label{intro: classifying stack}
For a reductive group $G$ over a perfect field $k$ of characteristic $p>0$ the Hodge-to-de Rham spectral sequence $E_1^{i,j}=H^j(BG, \Omega^i_{BG})\Rightarrow H^{i+j}_{\dR}(BG/k)$ of the classifying stack $BG$ degenerates at the first page.
\end{thm}

When $p=\mathrm{char}\,k$ is not a `torsion' prime for $G$ this has been proven by Totaro \cite[Theorem 0.2]{totaro}. For $p=2$, the degeneration has also been proven for special orthogonal groups $SO(n)$ by Totaro \cite[Theorem 11.1]{totaro}, for $G_2$ and $\Spin(n)$ with $n\leq 11$ by Primozic \cite{primozic}, and for $\PGL_{4m+2}, \PSO_{4m+2}, \PSp_{4m+2}$ for all $m$ by Kubrak-Scavia \cite{kubrak-scavia}. In all of these works the authors have been able to compute explicitly the de Rham and Hodge cohomology in question. Kubrak and Prikhodko proved in \cite[Theorem 1.3.23]{kubrak-prikhodko-modp} that the Hodge-to-de Rham spectral sequence of $BG$ for any reductive group $G$ over $k$ has no non-zero differentials coming out of entries with coordinates $(i,j)$ for $i+j\leq p-1$; our proof is similar to theirs in that we proceed by showing that the conjugate filtration on the de Rham complex of $BG$ is split.

In Section \ref{witt} we collect the preliminary material on de Rham cohomology, Witt vectors and quasi-$F$-splittings, in Section \ref{dr stack} we review the definition and basic properties of the de Rham stack, and prove the decomposition (\ref{intro: frob decomposition formula}) along with Theorem \ref{intro: main} for $F$-split varieties. In Section \ref{qfsplit: section} we prove Theorems \ref{intro: main} and \ref{intro: decomposition} in full generality. Section \ref{stacks} extends the decomposition of the Frobenius pullback of the de Rham complex to smooth Artin stacks, proves Frobenius splitting of $BG$, and deduces that the Hodge-to-de Rham spectral sequence of $BG$ degenerates at the $1$st page. In Section \ref{qrsp} we give another proof of the decomposition (\ref{intro: frob decomposition formula}) that proceeds by explicitly splitting the conjugate filtration on the Frobenius pullback of derived de Rham cohomology of quasiregular semiperfect algebras.% In Appendix \ref{continuity} we remark that 

{\bf Acknowledgements. } I am very grateful to Bhargav Bhatt for his suggestion, after hearing the argument in the $F$-split case, to consider the case of quasi-$F$-split varieties. Thanks to Vadim Vologodsky for the discussions of the alternative proofs of the presented results via the Azumaya algebra of differential operators. The result on cohomology of $BG$ arose from discussions with Dmitry Kubrak, to whom I am very thankful for this, as well as for comments on a draft of this text. I also would like to thank H\'el\`ene Esnault, Luc Illusie, Arthur Ogus, Gleb Terentiuk, Jakub Witaszek, and Shou Yoshikawa for useful conversations and comments.

The author was supported by the Clay Research Fellowship, and part of this research was conducted while the author was a member at the Institute for Advanced Study.

\section{Preliminaries}\label{witt}

In this section we introduce the notation for the de Rham complex, review some constructions related to Witt vectors, and recall the notion of the quotient by a quasi-ideal needed for our discussion of the de Rham stack. Throughout the paper, let $p$ be a fixed prime number.

\subsection{De Rham complex.} For a scheme $Y$ over $\bF_p$ we denote by $F_Y$ its absolute Frobenius endomorphism. For a morphism $f:X\to S$ of $\bF_p$-schemes we denote by $X^{(1)}$ the Frobenius twist of $X$, defined as the following fiber product:
\begin{equation}
\begin{tikzcd}
X^{(1)}\arrow[r]\arrow[d] & X\arrow[d, "f"]\\
S\arrow[r,"F_S"] & S
\end{tikzcd}
\end{equation}
We denote by $F_{X/S}:X\to X^{(1)}$ the relative Frobenius endomorphism, induced by the universal property of the fiber product from the absolute Frobenius $F_X:X\to X$. To lighten the notation, we will sometimes drop the subscript of $F_X$.

For a smooth morphism $f:X\to S$ let $\Omega^{\bullet}_{X/S}:=\cO_X\xrightarrow{d}\Omega^1_{X/S}\xrightarrow{d}\ldots$ be the relative de Rham complex, viewed as a complex of sheaves of $f^{-1}(\cO_S)$-modules on $X$. The differential on its pushforward under relative Frobenius $F_{X/S*}\Omega^{\bullet}_{X/S}$ respects the $\cO_{X^{(1)}}$-linear structure on the terms of this complex, and we view $F_{X/S*}\Omega^{\bullet}_{X/S}$ as an object of the derived category of quasi-coherent sheaves $D(X^{(1)})$ on $X^{(1)}$, cf. \cite[1.1]{deligne-illusie}.

\subsection{Animated rings.} The description of the de Rham stack that will be used for our arguments relies crucially on being able to view our scheme $X$ as a derived scheme, i.e. being able to consider its points with values in animated (also known as simplicial) rings. For a commutative ring $A$ we denote by $\CAlg^{\an}_A$ the $\infty$-category of animated commutative $A$-algebras. 

One source of animated rings needed for us is the quotient by a quasi-ideal in a classical ring. Recall, following \cite[3.3]{drinfeld-ring}, that for a commutative algebra $A$ a {\it quasi-ideal} in an $A$-algebra $B$ is a $B$-module $I$ and a map $d:I\to B$ of $B$-modules such that $d(x)\cdot y=d(y)\cdot x$ for all $x, y\in I$. The latter condition is equivalent to $I\xrightarrow{d}B$ forming a differential graded algebra with $I$ in degree $-1$ and $B$ in degree $0$, the graded algebra structure being defined by the algebra structure on $B$ and the $B$-module structure on $I$.

Given a quasi-ideal $I$ in an $A$-algebra $B$, we can form the quotient animated $A$-algebra $\cone(I\xrightarrow{d}B)\in \CAlg_A^{\an}$, see for instance \cite[3.6.3]{drinfeld-ring}. If $I$ is an invertible $B$-module, then the animated $A$-algebra $\cone(I\xrightarrow{d}B)$ is also described in \cite[Construction 2.1]{bhatt-lurie-prismatization}. If $d$ is injective, then this animated $A$-algebra is equivalent to the classical ring quotient $B/I$.

\subsection{Witt vectors.} In this subsection, let $R$ be an arbitrary $\bF_p$-algebra. We denote by $W(R)$ its ring of $p$-typical Witt vectors. It has a ring endomorphism $F:W(R)\to W(R)$ referred to as Frobenius, and an endomorphism of the underlying abelian group $V:W(R)\to W(R)$, called Verschiebung, such that $FV=VF=p$. We first make an observation about Frobenius on the Witt vectors on the level of classical rings and then generalize it to a construction involving animated rings, the latter being relevant for the discussion of the de Rham stack.

The restriction map $W(R)\to R$ factors through $W(R)/p$, inducing a surjective map that we denote by $r_R^{\cl}:W(R)/p\to R$. The reduction modulo $p$ of $F$ coincides with the intrinsic Frobenius endomorphism of the $\bF_p$-algebra $W(R)/p$. Since $FV=0$ on $W(R)/p$, this Frobenius endomorphism annihilates $\ker r_R^{\cl}=VW(R)/pW(R)$, which is to say that it factors as
\begin{equation}\label{witt intro: classical maps}
F:W(R)/p\xrightarrow{r_R^{\cl}}R\xrightarrow{s_R^{\cl}}W(R)/p
\end{equation}
for some $\bF_p$-algebra map $s_R^{\cl}$. The map $s_R^{\cl}$ sends an element $r\in R$ to the $p$-th power $[r]^p\in W(R)/p$ of its Teichmuller representative.

We will now construct analogs of $r^{\cl}_R$ and $s^{\cl}_R$ with $W(R)/p$ replaced by the {\it derived} modulo $p$ reduction of $W(R)$. We denote by $\cone(W(R)\xrightarrow{p}W(p))$ the animated commutative $W(R)$-algebra obtained by taking the quotient by the quasi-ideal $W(R)\xrightarrow{p}W(R)$. Equivalently, it is the derived tensor product $W(R)\otimes^L_{\bZ_p}\bF_p$ in animated commutative rings. 

We will also consider the animated commutative $W(R)$-algebra \begin{equation}\cone(F_*W(R)\xrightarrow{p}F_*W(R))\simeq F_*W(R)\otimes^L_{\bZ_p}\bF_p\end{equation} where $F_*W(R)$ denotes $W(R)$ viewed as an $W(R)$-algebra through the Witt vector Frobenius morphism $F:W(R)\to W(R)$.

\begin{rem}
Since the object of the derived category of $W(R)$-modules underlying $\cone(W(R)\xrightarrow{p}W(R))$ is the cofiber of the multiplication by $p$ map on $W(R)$, the animated ring $\cone(W(R)\xrightarrow{p}W(R))$ has non-trivial homotopy groups only in degrees $0$ and $1$. The ring $\pi_0(\cone(W(R)\xrightarrow{p}W(R)))$ is the quotient $W(R)$ by the ideal $pW(R)$ in the usual sense, and $\pi_1(\cone(W(R)\xrightarrow{p}W(R)))$ is identified with the $p$-torsion subgroup in $W(R)$.

The ring $W(R)$ is $p$-torsion-free if and only if $R$ is reduced\footnote{Multiplication by $p$ equals $VF$ on $W(R)$, the map $V$ is injective, and an element $[r_0]+V[r_1]+V^2[r_2]+\ldots\in W(R)$ is annihilated by $F$ if and only if $r_i^p=0$ for all i. So $\ker p=\ker F$ is non-zero if and only if $R$ has non-zero nilpotetns.}. In this case, $\cone(W(R)\xrightarrow{p}W(R))$ is the classical commutative ring $W(R)/p$.
\end{rem}

There is a natural map of animated $W(R)$-algebras $r_R:\cone(W(R)\xrightarrow{p} W(R))\to R$ induced by the map of quasi-ideals
\begin{equation}
\begin{tikzcd}
W(R)\arrow[d]\arrow[r,"p"] & W(R)\arrow[d] \\
0\arrow[r] & R
\end{tikzcd}
\end{equation}
where the right vertical map is the surjection given by the $0$-th Witt component.

Next, we will construct a natural map $s_R:R\to \cone(F_*W(R)\xrightarrow{p}F_*W(R))$ of animated $W(R)$-algebras whose composition $R\to \cone(F_*W(R)\xrightarrow{p}F_*W(R))\to F_*R$ with $F_*r_R$ is the Frobenius endomorphism of $R$. The map $s_R$ is induced by the map of quasi-ideals
\begin{equation}
\begin{tikzcd}
F_*W(R)\arrow[d,"\id"]\arrow[r,"V"] & W(R)\arrow[d, "F"] \\
F_*W(R)\arrow[r, "p"] & F_*W(R)
\end{tikzcd}
\end{equation}
The quotient by the quasi-ideal $F_*W(R)\xrightarrow{V}W(R)$ is indeed the classical ring $R$ because $V$ is an injective map.

\begin{rem}If $W(R)$ is $p$-torsion-free so that $\cone(F_*W(R)\xrightarrow{p}F_*W(R))$ is concentrated in degree $0$, the map $s_R:R\to F_*W(R)/p$ coincides with the map $s_R^{\cl}$ defined in (\ref{witt intro: classical maps}).
\end{rem}

In other words, the construction of the map $s_R$ shows that the structure of an animated $W(R)$-algebra on $\cone(F_*W(R)\xrightarrow{p}F_*W(R))$ naturally refines to a structure of an animated $R$-algebra: the structure map $W(R)\to \cone(F_*W(R)\xrightarrow{p}F_*W(R))$ factors as the composition $W(R)\to R\xrightarrow{s_R}\cone(F_*W(R)\xrightarrow{p}F_*W(R))$.

As mentioned above, when $R$ is reduced, $\cone(F_*W(R)\xrightarrow{p}F_*W(R))$ coincides with the classical quotient $F_*W(R)/p$. We will need the following flatness property of this quotient in the case of a smooth $R$.

\begin{lm}\label{witt intro: F flat}
For a smooth algebra $R$ over a perfect $\bF_p$-algebra $k$ the quotient $F_*W_n(R)/p$ is a finite projective $R$-module, for every $n$. Moreover, the module $F_*W(R)/p$ is flat over $R$.
\end{lm}

\begin{proof}
For the first assertion, the proof of \cite[Proposition 2.9]{kttwyy} endows $F_*W_n(R)/p$ with a finite filtration with graded pieces of the form $F^i_*R/F^{i-1}_*R$. Each of these $R$-modules is finite projective by a local computation, cf. \cite[Proposition 1.1.6]{brion-kumar}.

The second assertion is a consequence of a classical general fact: over a Noetherian ring $R$, the inverse limit $\lim\limits_n M_n$ of a system of finitely generated projective $R$-modules with surjective transition maps is flat. Indeed, $\lim\limits_n M_n$ is a direct summand of an infinite product $\prod\limits_{i=0}^{\infty}R$, because we can establish each $M_n$ as a direct summand of a finite free $R$-module in a way compatible with restriction maps.

Since $R$ is Noetherian, $\prod\limits_{i=0}^{\infty}R$ is flat: to check flatness it suffices to show that for every ideal $I\subset R$ the natural map $I\otimes_R \prod\limits_{i=0}^{\infty}R\to \prod\limits_{i=0}^{\infty}R$ is injective, but since $I$ is finitely presented as an $R$-module, this map can be identified with $\prod\limits_{i=0}^{\infty}I\to \prod\limits_{i=0}^{\infty}R$, which is injective. Hence $\lim\limits_n M_n$ is a direct summand of a flat module and is flat itself.

In particular, $F_*W(R)/p=\lim\limits_n F_*W_n(R)/p$ is flat over $R$.
\end{proof}

We will also need a global version of these constructions. Let $X$ be a reduced\footnote{the following constructions make perfect sense for any $\bF_p$-scheme but we put ourselves in a setting where $W(\cO_X)$ is $p$-torsion-free, to avoid the ambiguity between derived and classical quotient by $p$} scheme over $\bF_p$. Denote by $W(\cO_X)$ the sheaf of rings on the underlying topological space of $X$, whose value on an affine open $\Spec R\subset X$ is $W(R)$. The maps $r_R$ and $s_R$ introduced above give rise to the maps of sheaves of $W(\cO_X)$-algebras
\begin{equation}
s_X:\cO_X\to F_*W(\cO_X)/p\qquad r_X:W(\cO_X)/p\to \cO_X 
\end{equation}
We also consider the composition $s_{X,n}:\cO_X\xrightarrow{s_X}F_*W(\cO_X)/p\to F_*W_n(\cO_X)/p$ with the restriction map to the length $n$ Witt vectors, for every $n\geq 1$. This map of sheaves of algebras in particular endows $F_*W(\cO_X)/p$ and each $F_*W_n(\cO_X)/p$ with a structure of an $\cO_X$-module. 

We denote by $W_n(X)$ the nilpotent thickening of $X$ over $\bZ/p^n$ obtained by gluing together $\Spec W_n(R)$ for all affine open $\Spec R\subset X$. The map $s_{X,n}$ gives rise to a map of $\bF_p$-schemes \begin{equation}\label{witt intro: truncated witt frobenius map}s_{X,n}:W_n(X)\times_{\bZ/p^n}\bF_p\to X\end{equation} whose composition with the natural closed embedding $X=W_1(X)\hookrightarrow W_n(X)\times_{\bZ/p^n}\bF_p$ is the Frobenius on $X$.
%For $n\geq 1$ we denote by $W_n(X)$ the $\bZ/p^n$-scheme which is a nilpotent thickening of $X$ obtained by gluing affine schemes $\Spec W_n(R)$ over all affine opens $\Spec R\subset X$, cf. \cite[A.1]{langer-zink}. 

%We define $W(X)$ to be the $p$-adic formal scheme with the underlying topologically ringed space $(X,W(\cO_X))$, where the sheaf of rings $W(\cO_X)$ is equipped with the $p$-adic topology. We denote by $W(X)/p$ the mod $p$ reduction $W(X)\times_{\Spf \bZ_p}\Spec\bF_p$ of this formal scheme which is a nilpotent thickening of $X$. The morphism $s_X:\cO_X\to F_*W(\cO_X)/p$ induces a morphism \begin{equation}s_X:W(X)/p\to X\end{equation} of schemes that identifies $W(X)/p$ with $\Spec_{X}F_*W(\cO_X)/p$. Note that we completely ignore the $V$-adic topology on the sheaf of rings $W(\cO_X)/p$.

\subsection{\texorpdfstring{$F$}{}-splitting and quasi-\texorpdfstring{$F$}{}-splitting.}

Recall the following special class of algebraic varieties introduced by Mehta and Ramanathan in \cite{mehta-ramanathan}:

\begin{definition}
For a scheme $X$ over $\bF_p$ a {\it Frobenius splitting} is an $\cO_X$-linear map $\tau:F_*\cO_X\to \cO_X$ such that the composition $\cO_X\xrightarrow{f\mapsto f^p}F_*\cO_X\xrightarrow{\tau}\cO_X$ equals the identity on $\cO_X$. A scheme admitting a Frobenius splitting is called {\it Frobenius split}, or {\it $F$-split}.
\end{definition}

Yobuko introduced in \cite{yobuko} a generalization of this notion, which happens to have similarly strong cohomological consequences. % For an integer $n\geq 1$ denote by $s_{X,n}:\cO_X\to F_*W_n(\cO_X)/p$ the composition of $s_X$ with the map induced by the surjection $W(\cO_X)\to W_n(\cO_X)$.

\begin{definition}[{\hspace{1sp}\cite[Definition 4.1]{yobuko}, \cite[2.3]{kttwyy}}]
For a scheme $X$ over $\bF_p$ and an integer $n\geq 1$, a {\it $n$-quasi-$F$-splitting} is a map $\tau_n:F_*W_n(\cO_X)/p\to \cO_X$ of $\cO_X$-modules such that the composition $\cO_X\xrightarrow{s_{X,n}} F_*W_n(\cO_X)/p\xrightarrow{\tau_n}\cO_X$ is the identity on $\cO_X$. A scheme admitting $n$-quasi-$F$-splitting for some $n$ is called {\it quasi-$F$-split}.
\end{definition}
A $1$-quasi-$F$-splitting is the same notion as an $F$-splitting, and a $n$-quasi-$F$-splitting gives rise to a $n'$-quasi-$F$-splitting for any $n'>n$.

Yobuko proved \cite[Theorem 4.4]{yobuko} that for a perfect field $k$ a smooth quasi-$F$-split $k$-scheme $X$ admits a flat lift over $W_2(k)$. %non-smooth?
Combined with the results of \cite{deligne-illusie} this implies that for a smooth proper quasi-$F$-split variety over $k$ of dimension $\leq p$ the Hodge-to-de Rham spectral sequence degenerates, and the Kodaira-Akizuki-Nakano vanishing theorem holds. One of the main goals of the present paper is to show that this is also true for quasi-$F$-split varieties of arbitrary dimension.
 
\section{de Rham stack}\label{dr stack}
\newcommand{\FWR}{\cone(F_*W(R)\xrightarrow{p}F_*W(R))}
Our key tool is the stacky interpretation of de Rham cohomology of schemes in characteristic $p$. In this section we recall the definition and basic properties of the de Rham stack, following \cite{drinfeld-prismatization}, \cite{bhatt-lurie-prismatization}, \cite{bhatt}. Let $k$ be an arbitrary commutative $\bF_p$-algebra, it will serve as our base.

\begin{definition}\label{dr stack: affine base def}For a $k$-scheme $X$ we define its {\it de Rham stack} as a functor from $k$-algebras to groupoids, defined on a $k$-algebra $R$ by
\begin{equation}\label{dr prelim: def formula}
(X/k)^{\dR}:R\mapsto X(\FWR)
\end{equation}
\end{definition}

To evaluate the $k$-scheme $X$ on $\FWR$, we view it here as an animated $k$-algebra via the map $k\to R\xrightarrow{s_R}\FWR$. Recall \cite[Ch.2, \S 2.2.4]{toen-vezzosi} that for a classical $k$-scheme $X$ the value of $X$ on an animated $k$-algebra $A$ is defined by $\colim\limits_{B\to A}X(B)$ where the colimit is taken in the $\infty$-category of spaces over all classical commutative $k$-algebras $B$ equipped with a map to $A$.

We also abbreviate the notation for the absolute de Rham stack $(X/\bF_p)^{\dR}$ to $X^{\dR}$.

\begin{rem}\label{dr stack: definition remarks}
\begin{enumerate}
\item Since the animated ring $\FWR$ has non-zero homotopy groups only in degrees $0$ and $1$, the value of $X$ on it is indeed a $1$-groupoid, rather than an arbitrary space, by \cite[Ch. 2, Corollary 2.2.4.6]{toen-vezzosi}.

\item In \cite[Definition 2.5.3]{bhatt} the de Rham stack of $X$ relative to $k$ is defined by sending a test $k$-algebra $R$ to $X(\bG_a^{\dR}(R))$. Here $\bG_a^{\dR}$ is a ring stack defined in \cite[Definition 2.5.1]{bhatt}, and by \cite[Corollary 2.6.8]{bhatt} it coincides with the fppf-sheafification of the presheaf of animated $k$-algebras $R\mapsto F_*W(R)/p$. Since $H^{>0}_{\fppf}(\Spec R, W)=0$, this presheaf is already a sheaf, so the natural map of animated $k$-algebras $F_*W(R)/p\to \bG_a^{\dR}(R)$ is an equivalence, and Definition \ref{dr stack: affine base def} recovers the same stack.

\item In \cite[Lemma 7.3]{bhatt-lurie-prismatization} it is proven that relative prismatization, and consequently the de Rham stack, is a sheaf for fpqc topology, the main ingredient in the proof being the algebraization result \cite[Theorem 4.1]{bhatt-algebraization}. For the convenience of the reader we give in Lemma \ref{dr stack: etale descent} below a proof of the fact that $X^{\dR}$ is an \'etale sheaf, which amounts to observing that the relevant algebraization result in the special case of the de Rham stack follows directly from square-zero deformation theory.

We will use sheafyness of $X^{\dR}$ repeatedly to reduce some constructions to the case of an affine $X$.  %As explained in \cite[Construction 7.1]{bhatt-lurie-prismatization} the prismatization of a qcqs $p$-adic formal scheme satisfies \'etale descent, and consequently $(X/k)^{\dR}$ is a sheaf of groupoids for the \'etale topology on $k$-algebras. We will use the Zariski sheaf property of $(X/k)^{\dR}$ repeatedly.
\end{enumerate}
\end{rem}

\begin{lm}\label{dr stack: etale descent}
Let $R\to S$ be a fully faithful \'etale map of $\bF_p$-algebras. For any $\bF_p$-scheme $X$ the natural map \begin{equation}\label{dr stack: descent map}
  X(W(R)\otimes^L_{\bZ_p}\bF_p)\to\lim\limits_{\Delta}X(W(S^{\otimes_R\bullet})\otimes^L_{\bZ_p}\bF_p)
  \end{equation} is an equivalence.
\end{lm}

\begin{proof}
Let us first treat the case of an affine $X=\Spec A$. By \'etale descent we have an equivalence $R\simeq \lim\limits_{\Delta}S^{\otimes_R\bullet}$, and since the functor $R\mapsto W(R)\otimes_{\bZ_p}^L\bF_p$ from $\bF_p$-algebras to animated $\bF_p$-algebras preserves limits, we get an equivalence 
\begin{equation}\label{dr stack: wmodp descent}
W(R)\otimes_{\bZ_p}^L\bF_p\to \lim\limits_{\Delta}W(S^{\otimes_R\bullet})\otimes_{\bZ_p}\bF_p\end{equation} even though this limit diagram is no longer the \v{C}ech nerve of an \'etale $W(R)\otimes_{\bZ_p}^L\bF_p$-algebra in general. The assertion of the lemma for $X=\Spec A$ now amounts to the fact that the mapping spaces from $A$ to the LHS and RHS of (\ref{dr stack: wmodp descent}) are equivalent.

To prove the lemma for a general scheme $X$ we apply deformation theory to the square-zero extension $W(R)\otimes_{\bZ_p}^L\bF_p\to R$. Denote the ideal of this extension by $I_R\in D^{[-1,0]}(R)$. The map (\ref{dr stack: descent map}) fits into the commutative diagram
\begin{equation}\label{dr stack: descent map with nu}
\begin{tikzcd}
X(W(R)\otimes^L_{\bZ_p}\bF_p)\arrow[r]\arrow[d] & \lim\limits_{\Delta}X(W(S^{\otimes_R\bullet})\otimes^L_{\bZ_p}\bF_p)\arrow[d] \\
X(R) \arrow[r,"\sim"] & \lim\limits_{\Delta}X(S^{\otimes_R\bullet})
\end{tikzcd}
\end{equation}
where the bottom horizontal arrow is an equivalence by \'etale sheafyness of $X$. Hence it suffices to check that (\ref{dr stack: descent map}) becomes an equivalence after pullback to any point in $*\in X(R)$. Given such a point $x:\Spec R\to X$ the fiber of the right vertical map in (\ref{dr stack: descent map with nu}) is a torsor over $\lim\limits_{\Delta}\tau^{\leq 0}\RHom_{S^{\otimes_R \bullet}}(x^*L_{X/\bF_p}\otimes_R S^{\otimes_R\bullet}, I_{S^{\otimes_R\bullet}})\simeq \lim\limits_{\Delta}\tau^{\leq 0}\RHom_R(x^*L_{X/\bF_p}, I_R)$ while the fiber of the left vertical map is a torsor over $\tau^{\leq 0}\RHom_R(x^*L_{X/\bF_p}, I_R)$. Here $L_{X/\bF_p}$ is the cotangent complex of $X$.

As mentioned above, the natural map $W(R)\otimes_{\bZ_p}^L\bF_p\to \lim\limits_{\Delta}W(S^{\otimes_R\bullet})\otimes^L_{\bZ_p}\bF_p$ is an equivalence, in particular $I_R\to \lim\limits_{\Delta}I_{S^{\otimes_R \bullet}}$ is an equivalence. Therefore the map between fibers of the vertical map (\ref{dr stack: descent map with nu}) is an equivalence, as desired.
\end{proof}

The natural $R$-linear map $F_*r_R:\FWR\to F_*R$ induces a map $X(\FWR)\to X(F_*R)=X^{(1)}(R)$ which gives rise to a map of stacks \begin{equation}\nu_X:(X/k)^{\dR}\to X^{(1)}.\end{equation}
The natural map $s_R:R\to \FWR$ likewise induces a morphisms of stacks
\begin{equation}
\pi_X:X\to (X/k)^{\dR}
\end{equation}
By construction, the composition $\nu_X\circ \pi_X$ is the relative Frobenius morphism $F_{X/k}:X\to X^{(1)}$.

\begin{definition}For a morphism $f:X\to S$ of $\bF_p$-schemes with $S$ not necessarily affine, we define the {\it relative de Rham stack} as the fiber product in presheaves of groupoids on $\bF_p$-algebras:
\begin{equation}
(X/S)^{\dR}:=X^{\dR}\times_{S^{\dR}}S
\end{equation}
where $S$ maps to $S^{\dR}$ via the map $\pi_S$. 
\end{definition}

\begin{rem}
\begin{enumerate}

\item This is consistent with Definition \ref{dr stack: affine base def} in the case $S=\Spec k$ is affine, because for an animated $k$-algebra $A$ the space of $A$-points $X(A)=\Map_k (\Spec A, X)$ is equivalent to $$\Map_{\bF_p}(\Spec A, X)\times_{\Map_{\bF_p}(\Spec A, \Spec k)}*$$ where $*\to \Map_{\bF_p}(\Spec A, \Spec k)$ classifies the structure map $\Spec A\to \Spec k$.

\item If $S$ is a perfect $\bF_p$-scheme then the natural map $(X/S)^{\dR}\to X^{\dR}$ is an equivalence. This would follow from the fact that the natural map $\pi_S:S\to S^{\dR}$ is an equivalence. It suffices to show this in the case $S$ is an affine perfect scheme $\Spec k$. For a test $\bF_p$-algebra $R$ we have $(\Spec k)^{\dR}(R)=\Map_{\bF_p}(k,\FWR)$. In general, for an animated $\bF_p$-algebra $A$ there is a natural equivalence of mapping spaces between animated $\bF_p$-algebras \begin{equation}\label{dr stack: perfect adjunction}\Map_{\bF_p}(k, A)\simeq \Map_{\bF_p}(k, \pi_0(A)^{\perf})\end{equation} where $\pi_0(A)^{\perf}$ is the inverse limit perfection $\lim\limits_{\xleftarrow[\varphi]{}}\pi_0(A)$ of the ring $\pi_0(A)$. Indeed, in the diagram
\begin{equation}
\Map_{\bF_p}(k, A)\xleftarrow[\varphi_A\circ]{}\Map_{\bF_p}(k, A)\xleftarrow[\varphi_A\circ]{}\ldots
\end{equation}
all transition maps can be identified with $\circ \varphi_k$ and hence are equivalences, so its limit is equivalent to $\Map_{\bF_p}(k, A)$. On the other hand, its limit is tautologically identified with the mapping space $\Map_{\bF_p}(k, \lim\limits_{\xleftarrow[\varphi_A]{}} A)$. The perfection $A^{\perf}$ of an animated $\bF_p$-algebra is equivalent to the perfection of its $\pi_0$, by \cite[Corollary 11.9]{bhatt-scholze}, which gives (\ref{dr stack: perfect adjunction}).

Applying this observation to $A=\FWR$ we get $\Map_{\bF_p}(k,\FWR)\simeq \Map_{\bF_p}(k,R)$ because the Frobenius endomorphism of $\pi_0(\FWR)\simeq W(R)/p$ factors through the surjective map $r^{\cl}_R:W(R)/p\to R$, and therefore perfections of $W(R)/p$ and $R$ are identified.
\end{enumerate}
\end{rem}

The morphisms $\pi_X:X\to X^{\dR}$ and $\nu_X:X^{\dR}\to X$ then induce their relative versions
\begin{equation}
\pi_{X/S}:X\to (X/S)^{\dR} \qquad \nu_{X/S}:(X/S)^{\dR}\to X^{(1)}
\end{equation}
For smooth morphisms, cohomology of the structure sheaf on $(X/S)^{\dR}$ recovers the de Rham complex of $X$:

\begin{pr}\label{dr stack: cohomology}
For a smooth morphism $f:X\to S$ of $\bF_p$-schemes there is a natural quasi-isomorphism in $D(X^{(1)})$:
\begin{equation}\label{dr stack: cohomology formula}
R\nu_{X/S*}\cO_{(X/S)^{\dR}}\simeq F_{X/S*}\Omega^{\bullet}_{X/S}
\end{equation}
\end{pr}
\begin{proof}
In the case of an affine $S$ this is \cite[Corollary 2.7.2 (3)]{bhatt}. The proof in the general case can be deduced via Zariski descent, as we now outline. The de Rham complex $F_{X/S*}\Omega^{\bullet}_{X/S}$ can be identified with $\colim\limits_{j:U\hookrightarrow S}j_*F_{X|_U/U*}\Omega^{\bullet}_{X|_U/U}$ where the colimit in $D(X^{(1)})$ is taken over all affine open subschemes $j:U\hookrightarrow S$.

The formation of de Rham stack of a scheme commutes with fiber products. Therefore, for an open affine $U\hookrightarrow S$ we have $(X|_U)^{\dR}\simeq X^{\dR}\times_{S^{\dR}}U^{\dR}$, and the LHS of (\ref{dr stack: cohomology formula}) can be likewise expressed as $\colim\limits_{U\hookrightarrow S}j_*R\nu_{X|_U/U*}\cO_{(X|_U/U)^{\dR}}$. The natural equivalence (\ref{dr stack: cohomology formula}) for affine $S$ gives rise to an equivalence between functors $\{j:U\hookrightarrow S\}\to D(X^{(1)})$ to the derived $\infty$-category of quasi-coherent sheaves on $X^{(1)}$ given by $U\mapsto j_*R\nu_{X|_U/U*}\cO_{(X|_U/U)^{\dR}}$ and $U\mapsto j_*F_{X|_U/U*}\Omega^{\bullet}_{X|_U/U}$. Passing to colimits over these functors gives the desired equivalence.
%Alternatively, the calculation of the cohomology of the de Rham stack also follows from the results of \cite{bhatt-lurie-prismatization}.
\end{proof}

The key additional leverage for studying de Rham cohomology that the de Rham stack provides for us is the structure of a gerbe on the morphism $\nu_{X/S}$. For a vector bundle $E$ on a scheme $Y$, let us view its total space $\Tot_Y(E)\to Y$ as a group scheme over $Y$, fiberwise isomorphic to $\bG_a^{\rk E}$. We denote by $E^{\sharp}$ the divided power envelope of the zero section in $\Tot_Y(E)$, viewed as an affine group scheme over $Y$.

\begin{pr}
For a smooth morphism $f:X\to S$ the morphism $\nu_{X/S}:(X/S)^{\dR}\to X^{(1)}$ has a structure of a gerbe for the group scheme $T^{\sharp}_{X^{(1)}/S}$.
\end{pr}
\begin{proof}
In the case of an affine $S=\Spec k$ this is \cite[Proposition 2.7.1]{bhatt}. If $k$ is a perfect $\bF_p$-algebra, this also follows from \cite[Proposition 5.12]{bhatt-lurie-prismatization} applied to the prism $(A, I)=(W(k),(p))$ and the smooth scheme $X^{(1)}$ over $A/I\simeq k$. The case of a general $S$ follows by Zariski descent, as in Proposition \ref{dr stack: cohomology}: for each affine open $U\subset S$ we have a natural map $X^{(1)}|_U\to B^2_{X^{(1)}_U}T_{X^{(1)}_U/U}^{\sharp}$ classifying the gerbe $(X_U/U)^{\dR}\to X_U^{(1)}$. By Zariski descent they define a map $X^{(1)}\to B^2_{X^{(1)}} T_{X^{(1)}/S}^{\sharp}$ such that the pullback of the universal gerbe $X^{(1)}\to B^2_{X^{(1)}} T_{X^{(1)}/S}^{\sharp}$ along it is equivalent to $(X/S)^{\dR}\to X^{(1)}$.
\end{proof}

Given this gerbe structure, we can readily deduce that the de Rham complex decomposes after Frobenius pullback. This has been previously observed by Bhargav Bhatt, and by Vadim Vologodsky via a different method relying on the Azumaya property of the algebra of differential operators.

\begin{pr}[{\hspace{1sp}\cite[Remark 2.7.5]{bhatt}}]\label{dr stack: decomposition after frobenius}
For a smooth morphism $f:X\to S$ there is a natural quasi-isomorphism in $D(X)$
\begin{equation}\label{dr stack: decomposition after frobenius formula}
F_{X/S}^*F_{X/S*}\Omega^{\bullet}_{X/S}\simeq \bigoplus\limits_{i\geq 0}F^*_{X/S}\Omega^{i}_{X^{(1)}/S}[-i]
\end{equation}
that induces the Frobenius pullback of the Cartier isomorphism on all cohomology sheaves.
\end{pr}
\begin{proof}
Since the composition $X\xrightarrow{\pi_{X/S}}(X/S)^{\dR}\xrightarrow{\nu_{X/S}}X^{(1)}$ is equal to the relative Frobenius morphism $F_{X/S}$, the pullback $\nu'_{X/S}$ of the morphism $\nu_{X/S}$ along $F_{X/S}$ admits a section:
\begin{equation}\label{dr stack: F pullback diagram}
\begin{tikzcd}
(X/S)^{\dR}\times_{X^{(1)}} X\arrow[r]\arrow[d, "\nu'_{X/S}"] & (X/S)^{\dR} \arrow[d,"\nu_{X/S}"] \\
X\arrow[r, "F_{X/S}"]\arrow[u, bend left=30, "{(\pi_{X/S},\id_X)}"] & X^{(1)}
\end{tikzcd}
\end{equation}
As $(X/S)^{\dR}$ is a gerbe for $T_{X^{(1)}/S}^\sharp$ over $X^{(1)}$, the pullback $(X/S)^{\dR}\times_{X^{(1)}}X$ is a gerbe for the group scheme $(F^*_{X/S}T_{X^{(1)}/S})^{\sharp}$. Since the morphism $\nu_{X/S}'$ has a section, this gerbe is trivial, so the stack $(X/S)^{\dR}\times_{X^{(1)}}X$ is isomorphic to the relative classifying stack $B_X(F^*_{X/S}T_{X^{(1)}/S})^\sharp$.

For a vector bundle $E$ on $X$ the derived pushforward of the structure sheaf along $B_XE^\sharp\to X$ is identified with $\bigoplus\limits_{i\geq 0}\Lambda^i E^{\vee}[-i]$, by \cite[Lemma 7.8]{bhatt-lurie-prismatization} or \cite[Remark 2.4.6]{bhatt}. Therefore, we have an equivalence
\begin{equation}\label{dr stack: split gerbe cohomology formula}
R\nu'_{X/S*}\cO_{(X/S)^{\dR}\times_{X^{(1)}}X}\simeq \bigoplus\limits_{i\geq 0}F_{X/S}^*\Omega^i_{X^{(1)}/S}[-i]
\end{equation}
On the other hand, pushforward along $\nu_{X/S}$ satisfies flat base change by Lemma \ref{proof: gerbe base change} below. That is, $R\nu'_{X/S*}\cO_{(X/S)^{\dR}\times_{X^{(1)}}X}$ is equivalent to $F_{X/S}^*R\nu_{X/S*}\cO_{(X/S)^{\dR}}$ which in turn is identified with $F_{X/S}^*F_{X/S*}\Omega^{\bullet}_{X/S}$ by Proposition \ref{dr stack: cohomology}. Hence the equivalence (\ref{dr stack: split gerbe cohomology formula}) provides the desired decomposition.

The final assertion about the effect of (\ref{dr stack: decomposition after frobenius formula}) on cohomology sheaves can always be ensured by composing with an appropriate automorphism of the RHS.
\end{proof}
\begin{lm}[{\hspace{1sp}\cite[Remark 2.5.5]{bhatt}}]\label{proof: gerbe base change}
Let $X$ be a smooth $S$-scheme, and $f:Y\to X^{(1)}$ be a flat morphism. Consider the fiber square
\begin{equation}
\begin{tikzcd}
Y\times_{X^{(1)}} (X/S)^{\dR}\arrow[r]\arrow[d,"\nu'_{X/S}"] & (X/S)^{\dR}\arrow[d, "\nu_{X/S}"] \\
Y\arrow[r,"f"] & X^{(1)}
\end{tikzcd}
\end{equation}
The pullback $f^*F_{X/S*}\Omega^{\bullet}_{X/S}$ of the de Rham complex is naturally identified with $R(\nu'_X)_*\cO_{Y\times_X X^{\dR}}$.
\end{lm}

Recall that for a morphism $f:X\to S$ of $\bF_p$-schemes we say that $X$ is {\it Frobenius-split relative to} $S$ if there exists an $\cO_{X^{(1)}}$-linear map $\tau:F_{X/S*}\cO_X\to \cO_{X^{(1)}}$ such that the composition $\cO_{X^{(1)}}\xrightarrow{F_{X/S}^{\#}}F_{X/S*}\cO_{X/S}\xrightarrow{\tau}\cO_{X^{(1)}}$ is the identity map. Proposition \ref{dr stack: decomposition after frobenius} implies that a smooth Frobenius split scheme has decomposable de Rham complex:

\begin{pr}\label{dr stack: relative fsplit decomposition}
If a smooth $S$-scheme $X$ is Frobenius-split relative to $S$, then there is a quasi-isomorphism $F_{X/S*}\Omega^{\bullet}_{X/S}\simeq \bigoplus\limits_{i\geq 0}\Omega^i_{X^{(1)}/S}[-i]$ in $D(X^{(1)})$.
\end{pr}
\begin{proof}
Applying pushforward $F_{X/S*}$ along the finite flat morphism $F_{X/S}:X\to X^{(1)}$ to the quasi-isomorphism (\ref{dr stack: decomposition after frobenius formula}) we obtain a quasi-isomorphism in $D(X^{(1)})$:
\begin{equation}\label{dr stack: decomposition after tensor with fpush}
F_{X/S*}\cO_{X}\otimes_{\cO_{X^{(1)}}}F_{X/S*}\Omega^{\bullet}_{X/S}\simeq F_{X/S*}\cO_X\otimes_{\cO_{X^{(1)}}}\bigoplus\limits_{i\geq 0}\Omega^i_{X^{(1)}/S}[-i].
\end{equation}
Given a Frobenius splitting $\tau: F_{X/S*}\cO_X\to \cO_{X^{(1)}}$ we will define the desired quasi-isomorphism as the composition
\begin{multline}\label{dr stack: f split decomposition formula}
F_{X/S*}\Omega^{\bullet}_{X/S}\xrightarrow{F^{\#}_{X/S}\otimes\id}F_{X/S*}\cO_{X}\otimes_{\cO_{X^{(1)}}}F_{X/S*}\Omega^{\bullet}_{X/S}\simeq \\ \simeq F_{X/S*}\cO_X\otimes_{\cO_{X^{(1)}}}\bigoplus\limits_{i\geq 0}\Omega^i_{X^{(1)}/S}[-i]\xrightarrow{\tau\otimes\id}\bigoplus\limits_{i\geq 0}\Omega^i_{X^{(1)}/S}[-i]
\end{multline}
By the final sentence of Proposition \ref{dr stack: decomposition after frobenius} the map induced by quasi-isomorphism (\ref{dr stack: decomposition after tensor with fpush}) on cohomology sheaves is obtained by extending scalars from $\cO_{X^{(1)}}$ to $F_{X/S*}\cO_{X}$. This implies that the composition (\ref{dr stack: f split decomposition formula}) induces an isomorphism on cohomology sheaves. 
\end{proof}

\section{Quasi-\texorpdfstring{$F$}{}-split varieties}\label{qfsplit: section}

We will now strengthen Proposition \ref{dr stack: decomposition after frobenius} to prove that de Rham complex of a smooth variety splits already after being pulled back along the morphism $s_{X,n}:W_n(X)\times_{\bZ/p^n}\bF_p\to X$, for every $n$, and consequently de Rham complex of a smooth quasi-$F$-split variety decomposes.

We restrict ourselves to working over a perfect $\bF_p$-algebra $k$. For a smooth $k$-scheme $X$ we denote by $F_*\Omega^{\bullet}_X$ the de Rham complex, viewed as a complex of coherent sheaves on $X$, dropping the perfect base ring $k$ from the notation. For for every $n\geq 1$ we view $F_*W_n(\cO_X)/p$ as a sheaf of algebras on $X^{(1)}$ endowed with a structure of an $\cO_{X^{(1)}}$-algebra via the map $s_{X, n}:\cO_{X^{(1)}}\to F_*W_n(\cO_X)/p$ given by $\cO_X\otimes_{k,\varphi_k}k\ni r\otimes 1\mapsto [r^p]$ on local sections. As in (\ref{witt intro: truncated witt frobenius map}), we denote by the same symbol $s_{X,n}:W_n(X)\times_{\bZ/p^n}\bF_p\to X^{(1)}$ the corresponding map of $k$-schemes.
 
\begin{thm}\label{qfsplit: main decomposition}
For a smooth scheme $X$ over a perfect $\bF_p$-algebra $k$, for each $n$ there is a natural equivalence
\begin{equation}\label{qfsplit: main formula}
F_{*}W_n(\cO_X)/p\otimes_{\cO_{X^{(1)}}}F_{*}
\Omega^{\bullet}_{X}\simeq\bigoplus\limits_{i\geq 0}F_*W_n(\cO_X)/p\otimes_{\cO_{X^{(1)}}}\Omega^i_{X^{(1)}}[-i]
\end{equation}
in the derived category of quasi-coherent $F_*W_n(\cO_X)/p$-modules on $X^{(1)}$.
\end{thm}

We will obtain this decomposition as a consequence of the fact that the morphism $\nu_X:X^{\dR}\to X^{(1)}$ acquires a section over $W_n(X)\times_{\bZ_p}\bF_p$. We first recall the following criterion for the morphism $\nu_X$ itself to have a section.

\begin{pr}\label{dr stack: deltalift-splitting}
Let $\tX$ be a flat $p$-adic formal scheme over $\Spf\bZ_p$ equipped with an endomorphism $\tF:\tX\to \tX$ whose restriction to the special fiber $X:=\tX\times_{\bZ_p}\bF_p$ is equal to the Frobenius endomorphism of $X$. Then the morphism $\nu_X:X^{\dR}\to X$ admits a natural section $s_{\tF}:X\to X^{\dR}$.
\end{pr}

\begin{proof}
This is established in the course of the proof of \cite[Proposition 5.12]{bhatt-lurie-prismatization}, but we record the proof here for the sake of completeness. 

Since $X^{\dR}$ is a Zariski sheaf (Lemma \ref{dr stack: etale descent}), it suffices to produce a section on every affine Zariski open $U\hookrightarrow X$ provided that this construction is functorial, in the sense that it extends to a functor from the poset of affine opens $\{U\subset X\}$ to the slice category over $X^{\dR}$. For this it is enough to produce a section of $\nu_U:U^{\dR}\to U$ for every affine open $U\subset X$. Each $U$ is equipped with a unique flat formal lift over $\bZ_p$ together with a Frobenius lift, compatible with the given lift of $X$. Hence we may assume from now on that $X=\Spec S$ is affine.

Let $\tS$ be a flat $\bZ_p$-algebra equipped with an endomorphism $\tF_S:\tS\to \tS$ that lifts the Frobenius endomorphism on $S:=\tS/p$. There is a natural ring homomorphism $\tw_{\tF}:\tS\to W(\tS)$ whose composition with the projection $W(\tS)\to \tS$ on the $0$-th Witt coordinate is the identity map on $\tS$, by the adjunction between $W(-)$ and the forgetful functor from $\delta$-rings to all rings, cf. \cite[Th\'eor\`eme 4]{joyal}. In particular, we get a natural map of rings
\begin{equation}
   w_{\tF}: \tS\xrightarrow{\tw_{\tF}} W(\tS)\to W(S)
\end{equation} where the second map is induced by the mod $p$ reduction map $\tS\to S$. Applying the functor $-\otimes^L_{\bZ_p}\bF_p$ to the map $w_{\tF}$ we obtain a point of the groupoid $\Map_{\bF_p}(S, W(S)\otimes^L_{\bZ_p}\bF_p)=(\Spec S)^{\dR}(S)$ and we define the section $s_{\tF}:\Spec S\to (\Spec S)^{\dR}$ to be the map corresponding to that point.

This is indeed a section, because the composition $\nu_{\Spec S}\circ s_{\tF}$ is induced by the map $S\to S$ of (discrete) $\bF_p$-algebras given by the composition
\begin{equation}\label{dr stack: deltalift section composition}
\tS\otimes_{\bZ_p}\bF_p\xrightarrow{\tw_{\tF}}W(\tS)\otimes_{\bZ_p}^L\bF_p\to W(S)\otimes^{L}_{\bZ_p}\bF_p\xrightarrow{r_S}S
\end{equation}
The composition of the last two maps in (\ref{dr stack: deltalift section composition}) is equivalent to $W(\tS)\otimes^L_{\bZ_p}\bF_p\to \tS\otimes^L_{\bZ_p}\bF_p\simeq S$, the mod $p$ reduction of the projection onto $0$-th Witt coordinate. Hence the composition (\ref{dr stack: deltalift section composition}) is the identity map, because the composition $\tS\xrightarrow{\tw_{\tF}} W(\tS)\to \tS$ is the identity.
\end{proof}

\begin{rem}
Bhargav Bhatt and Vadim Vologodsky proved that for a smooth scheme $X$ over a perfect field $k$ the map $\nu_X:(X/k)^{\dR}\to X^{(1)}$ admits a section if and only if $X$ can be lifted over $W_2(k)$ together with its Frobenius endomorphism, but we will not need this stronger fact.
\end{rem}

For an $\bF_p$-scheme $X$ denote by $W(X)$ the ind-scheme obtained as the colimit $\colim\limits_n W_n(X)$ along closed embeddings $W_n(X)\hookrightarrow W_{n+1}(X)$ of finite length Witt schemes. Its base change $W(X)\times_{\bZ_p}\bF_p$ can be likewise described as the colimit $\colim\limits_n W_n(X)\times_{\bZ/p^n}\bF_p$. The maps $\cO_X\to W_n(\cO_X)/p$ of sheaves of rings on $X$ given on local sections by $r\mapsto [r^p]$ give rise to a map $s_X:W(X)\times_{\bZ_p}\bF_p\to X$ such that the composition $X\simeq W_1(X)\hookrightarrow W(X)\times_{\bZ_p}\bF_p\xrightarrow{s_X}X$ is the Frobenius endomorphism of $X$.

\begin{cor}\label{dr stack: section on W mod p}
For a reduced $\bF_p$-scheme $X$ there is a natural morphism $\sigma: W(X)\times_{\bZ_p}\bF_p\to X^{\dR}$ fitting into the commutative diagram
\begin{equation}
\begin{tikzcd}
& X^{\dR}\arrow[d, "\nu_X"] \\
W(X)\times_{\bZ_p}\bF_p\arrow[ru, "\sigma"]\arrow[r, "s_X"] & X.
\end{tikzcd}.
\end{equation}
\end{cor}
\begin{proof}
Note that $W(X)$ is in general not a $p$-adic formal scheme because for an affine open $\Spec R\subset X$ the ring $W(R)$ is complete for the $V$-adic topology that is stronger than its $p$-adic topology. We make the following construction that lets us ignore the $V$-adic topology on the mod $p$ reduction of $W(R)$. Consider the Zariski sheaf of sets on the category of $\bF_p$-algebras obtained as the colimit \begin{equation}\bW(X)_{\bF_p}:=\colim\limits_{\Spec R\hookrightarrow X}\Spec W(R)/p\end{equation} taken over the category of all affine open subschemes of $X$. Each scheme $W_n(X)\times_{\bZ/p^n}\bF_p$, being the colimit (in Zariski sheaves) of $\Spec W_n(R)/p$ for affine opens $\Spec R\hookrightarrow X$, admits a natural map $W_n(X)\times_{\bZ/p^n}\bF_p\to \bW(X)_{\bF_p}$. These maps for varying $n$ give rise to a map $W(X)\times_{\bZ_p}\bF_p\to \bW(X)_{\bF_p}$ that can be interpreted as the map from the $V$-adic formal completion of the target.

For each affine open $\iota:\Spec R\hookrightarrow X$, consider the commutative diagram induced by the functoriality of the de Rham stack and the map $\nu_X$:
\begin{equation}
\begin{tikzcd}
(\Spec W(R)/p)^{\dR}\arrow[rr,"{(\iota\circ s_R)^{\dR}}"]\arrow[d, "{\nu_{W(R)/p}}"] & & X^{\dR}\arrow[d, "{\nu_X}"] \\
\Spec W(R)/p\arrow[r,"{s_R}"]\arrow[u, bend left=30,"{s_{F_{W(R)}}}"] & \Spec R\arrow[r,"\iota"] & X
\end{tikzcd}
\end{equation}

The map $\nu_{W(R)/p}$ has a section $s_{F_{W(R)}}$, natural in $R$, by Proposition \ref{dr stack: deltalift-splitting} applied to the $p$-adic formal lift $\Spf W(R)$ of $\Spec W(R)/p$ with the Frobenius lift given by the Witt vector Frobenius. The compositions $(\iota\circ s_R)^{\dR}\circ s_{F_{W(R)}}$ for varying $R$ then give rise to a map $\bW(X)_{\bF_p}=\colim\limits_{\Spec R\hookrightarrow X} \Spec W(R)/p\to X^{\dR}$ and its composition with the map $W(X)\times_{\bZ_p}\bF_p\to \bW(X)_{\bF_p}$ is our desired section $\sigma$.
\end{proof}
We can now deduce Theorem \ref{qfsplit: main decomposition} as in the proof of Proposition \ref{dr stack: decomposition after frobenius}.

\begin{proof}[Proof of Theorem \ref{qfsplit: main decomposition}]
Consider the pullback of the de Rham stack onto $W_n(X)_{\bF_p}:=W_n(X)\times_{\bZ_p}\bF_p$:
\begin{equation}
\begin{tikzcd}
X^{\dR}\times_{X} W_n(X)_{\bF_p}\arrow[d, "\nu'_X"]\arrow[r] & X^{\dR}\arrow[d, "\nu_X"] \\
W_n(X)_{\bF_p}\arrow[r,"s_{X,n}"] & X^{(1)}
\end{tikzcd}
\end{equation}
Corollary \ref{dr stack: section on W mod p} provides us with a section of the morphism $\nu_X'$. Since $\nu'_X$ is a gerbe for the group scheme $(s_{X,n}^*T_X)^{\sharp}$, a section provides a trivialization \begin{equation}X^{\dR}\times_{X} W_n(X)_{\bF_p}\simeq B_{W_n(X)_{\bF_p}}(s_{X,n}^*T_{X^{(1)}})^{\sharp}\end{equation} Hence $R\nu'_{X*}\cO$ is equivalent to $\bigoplus\limits_{i\geq 0}s_{X,n}^*\Omega^i_{X^{(1)}}[-i]$. 

By Lemma \ref{witt intro: F flat} the morphism $s_{X,n}:W_n(X)_{\bF_p}\to X^{(1)}$ is flat, and the base change along it (Lemma \ref{proof: gerbe base change}) gives an equivalence
\begin{equation}
R\nu'_{X*}\cO\simeq s_{X,n}^*R\nu_{X*}\cO_{X^{\dR}}\simeq s_{X,n}^*F_*\Omega^{\bullet}_{X}
\end{equation}
which gives rise to the decomposition $s_{X,n}^*F_*\Omega^{\bullet}_{X}\simeq \bigoplus\limits_{i\geq 0}s_{X,n}^*\Omega^i_{X^{(1)}}[-i]$. Since $s_{X,n}:W_n(X)_{\bF_p}\to X^{(1)}$ is an affine morphism with $s_{X,n*}\cO_{W_n(X)/p}\simeq F_{*}W_n(\cO_X)/p$, this is equivalent to the desired decomposition (\ref{qfsplit: main formula}).
\end{proof}

By construction, the above splittings for varying $n$ are compatible with restriction maps. Passing to the inverse limit over $n$, using that $F_{*}\Omega^{\bullet}_X$ is a bounded complex of coherent sheaves so tensoring with it commutes with inverse limits, we get:
\begin{cor}\label{qfsplit: limit decomposition}
For a smooth scheme $X$ over a perfect $\bF_p$-algebra $k$ there is a natural quasi-isomorphism
\begin{equation}\label{qfsplit: limit decomposition formula}
F_{*}W(\cO_X)/p\otimes_{\cO_{X^{(1)}}}F_{*}\Omega^{\bullet}_X\simeq \bigoplus\limits_{i\geq 0}F_{*}W(\cO_X)/p\otimes_{\cO_{X^{(1)}}}\Omega^i_{X^{(1)}}[-i]
\end{equation}
in the derived category of quasi-coherent $F_{*}W(\cO_X)/p$-modules on $X$.
\end{cor}

We can now deduce the first main result of this paper by repeating the argument from Proposition \ref{dr stack: relative fsplit decomposition}.

\begin{cor}\label{qfsplit: qfsplit decomposable}
For a quasi-$F$-split smooth scheme $X$ over a perfect $\bF_p$-algebra $k$ there is a quasi-isomorphism $F_{*}\Omega^{\bullet}_{X}\simeq\bigoplus\limits_{i\geq 0}\Omega^i_{X^{(1)}}[-i]$.
\end{cor}

\begin{proof}
By definition, for some $n$ there exists an $\cO_{X^{(1)}}$-linear map $\tau_n: F_*W_n(\cO_X)/p\to \cO_X^{(1)}$ splitting the map $\cO_{X^{(1)}}\to F_*W_n(\cO_X)/p$.

We may and do assume that the $F_{*}W_n(\cO_X)/p$-linear equivalence (\ref{qfsplit: main formula}) induces on each cohomology sheaf the base change to $F_{*}W(\cO_X)/p$ of the Cartier isomorphism $\cH^i(F_{*}\Omega^{\bullet}_X)\simeq \Omega^i_{X^{(1)}}$. If this is not true on the nose, we can always replace the equivalence (\ref{qfsplit: main formula}) with its composition with an appropriate automorphism of the RHS to make this true.

The natural map $s_{X,n}\otimes\id:F_{*}\Omega^{\bullet}_{X}\to F_*W_n(\cO_X)/p\otimes_{\cO_{X^{(1)}}}F_{*}\Omega^{\bullet}_{X}$ in $D(X)$ admits a splitting $\tau_n\otimes\id: F_*W(\cO_X)/p\otimes_{\cO_{X^{(1)}}}F_{*}\Omega^{\bullet}_{X}\to F_{*}\Omega^{\bullet}_{X}$. Composing the natural map $\bigoplus\limits_{i\geq 0}\Omega^i_{X^{(1)}}[-i]\to \bigoplus\limits_{i\geq 0} F_*W_n(\cO_X)/p\otimes_{\cO_{X^{(1)}}}\Omega^i_{X^{(1)}}[-i]$ with the inverse of the equivalence (\ref{qfsplit: main formula}) and the map $\tau_n\otimes \id$ we get a map in $D(X^{(1)})$:
\begin{equation}\label{qfsplit: dec map}
\bigoplus\limits_{i\geq 0}\Omega^i_{X^{(1)}}[-i]\to F_{*}\Omega^{\bullet}_{X}
\end{equation}
By our assumption on the effect of (\ref{qfsplit: main formula}) on cohomology sheaves, (\ref{qfsplit: dec map}) induces the Cartier isomorphism on cohomology sheaves, and in particular is a quasi-isomorphism, as desired.
\end{proof}

\begin{rem}
By \cite[Th\'eor\`eme 3.5]{deligne-illusie} splittings of $\tau^{\leq 1}F_*\Omega^{\bullet}_k$ are in bijection with isomorphism classes of lifts of $X^{(1)}$ to $W_2(k)$, so the equivalence in Corollary \ref{qfsplit: qfsplit decomposable} gives rise to a lift over $W_2(k)$ for any smooth $k$-scheme equipped with a quasi-$F$-splitting. In \cite[7.2]{kttwyy} a direct construction of a $W_2(k)$-lift of any (not necessarily smooth) quasi-$F$-split scheme over a perfect field $k$ is given. We expect these two constructions to give the same lifts for smooth quasi-$F$-split $k$-schemes, but do not pursue this here.
\end{rem}

\begin{rem}
In view of the decomposition (\ref{qfsplit: limit decomposition formula}), the proof of Corollary \ref{qfsplit: qfsplit decomposable} would go through given only that the map $\cO_{X^{(1)}}\to F_*W(\cO_X)/p$ into the sheaf of infinite length Witt vectors modulo $p$ admits a section, which might appear to be a weaker property than being quasi-$F$-split. However, we will see below in Corollary \ref{continuity: qf splitting} that for a smooth finite type scheme $X$ over a perfect field $k$ a section of $\cO_{X^{(1)}}\to F_*W(\cO_X)/p$ automatically factors through an $n$-quasi-$F$-splitting for some $n\geq 1$.
\end{rem}

Let us mention explicitly the following classes of examples of $F$-split and quasi-$F$-split varieties whose de Rham complex is therefore decomposable. This answers some of the questions raised in \cite[7.11]{illusie-hodge}:

\begin{cor}\label{qfsplit: flag varieties}
For the following classes of smooth varieties over a perfect field $k$ the de Rham complex is decomposable:
\begin{enumerate}
\item Generalized flag varieties $G/P$, where $P$ is a smooth parabolic subgroup in a reductive group $G$. In particular, every smooth quadric hypersurface $Q\subset\bP^{n+1}_k$ in a projective space has decomposable de Rham complex.

\item Finite height Calabi-Yau varieties, i.e. smooth proper connected varieties $X$ with $\omega_X\simeq\cO_X$, $H^i(X,\cO_X)=0$ for $0<i<\dim X$, and $H^{\dim X}(X, W(\cO_X))\otimes\bQ\neq 0$.
\end{enumerate}

\end{cor}
\begin{proof}
Generalized flag varieties are Frobenius-split by \cite[Theorem 2]{mehta-ramanathan}, which implies the first assertion. The assertion about smooth quadrics now follows, because they are examples of generalized flag varieties for orthogonal groups $SO(n)$.

Alternatively, we can give a direct construction of a Frobenius splitting of a smooth quadric $Q\subset\bP^n_k$. By \cite[Excercies 1.3.E (1)+(3)]{brion-kumar} a homogeneous polynomial $\sigma\in k[x_0,\ldots,x_n]$ of degree $(p-1)(n+1)$ defines a splitting of a hypersurface in $\bP^n$ cut out by equation $\{f=0\}$ if $f^{p-1}$ divides $\sigma$, and $(x_0x_1\ldots x_n)^{p-1}$ appears in $f^{p-1}$ with a non-zero coefficient.

Since the property of being Frobenius split can be checked after an extension of the base perfect field, we may assume that $k$ is algebraically closed. If $p$ is odd, then $Q$ is isomorphic to the quadric cut out by the equation $f=x_0^2+\ldots+x_n^2$. The coefficient of $(x_0x_1)^{p-1}$ in $f^{p-1}$ is $\binom{p-1}{({p-1})/{2}}\neq 0$, so we can take $\sigma=f^{p-1}\cdot (x_2\ldots x_n)^{p-1}$ as a degree $(n+1)(p-1)$ polynomial defining a splitting of $Q$. For $p=2$ the quadric $Q$ is isomorphic (\hspace{1sp}\cite[Expose XII, Proposition 1.2]{sga7}) to one of the following two: $f=x_0x_1+x_2x_3+\ldots+x_{n-1}x_n$ or $f=x_0^2+x_1x_2+x_3x_4+\ldots+x_{n-1}x_n$, depending on the parity of $n$. In either of the cases $\sigma=f\cdot (x_0x_1\ldots x_{n-2})$ defines a splitting.

Finally, the fact that finite height Calabi-Yau varieties are quasi-$F$-split was proven by Yobuko \cite[Theorem 4.5]{yobuko}.
\end{proof}

\begin{cor}\label{qfsplit: htdgr akn}
Let $X$ be a smooth proper quasi-$F$-split variety over a perfect field $k$, equidimensional of dimension $d$.
\begin{enumerate}
\item Hodge-to-de Rham spectral sequence $E_1^{i,j}=H^j(X,\Omega^i_{X/k})\Rightarrow H^{i+j}_{\dR}(X/k)$ degenerates at the first page.
\item For an ample line bundle $L$ on $X$, the cohomology group $H^i(X,\Omega^j_{X/k}\otimes L)$ vanishes for $i+j>d$.
\end{enumerate}
\end{cor}
\begin{proof}
(1) follows from Corollary \ref{qfsplit: qfsplit decomposable} by \cite[Corollaire 4.14]{deligne-illusie}. (2) follows by the argument from the proof of \cite[Corollaire 2.8]{deligne-illusie}: by Serre's vanishing the groups $H^i(X,\Omega^j\otimes L^N)$ are zero for large enough $N$ and $i>0$, hence Lemma 2.9 of loc. cit. implies that $H^i(X,\Omega^j\otimes L)=0$, which is equivalent to (2) by Serre duality.
\end{proof}

\begin{rem}
\begin{enumerate}
\item For $d\leq p$ Corollary \ref{qfsplit: htdgr akn} follows directly from the results \cite[Corollaire 2.4, Corollaire 2.8]{deligne-illusie} of Deligne-Illusie-Raynaud, by the result of Yobuko that a quasi-$F$-split variety lifts over $W_2(k)$. The case $d=p+1$ for an $F$-split variety has also been proven previously by Achinger-Suh \cite[Theorems A.2, A.4]{achinger-suh}.
\item Kodaira vanishing for quasi-$F$-split varieties of arbitrary dimension (the case $j=d$ in (2) above) has been proven in \cite[Theorem 1.7]{nakkajima-yobuko}, and for $F$-split varieties it was proven in \cite{mehta-ramanathan}. The above result in full generality appears to be new even for $F$-split varieties.
\end{enumerate}
\end{rem}

\section{Frobenius splitting of classifying stacks}\label{stacks}

In this section we generalize the decomposition result from Proposition \ref{dr stack: decomposition after frobenius} to smooth Artin stacks over a perfect field and observe that the classifying stack of a reductive group is Frobenius split, which leads to the main result of this section:

\begin{thm}\label{stacks: main htdr bg} For a reductive group $G$ over a perfect field $k$ of characteristic $p$, Hodge-to-de Rham spectral sequence $E_{1}^{i,j}=H^j(BG, \Lambda^i L_{BG/k})\Rightarrow H^{i+j}_{\dR}(BG/k)$ for the classifying stack of $G$ degenerates at the first page.
\end{thm}

Here $\Lambda^i L_{BG/k}$ is the $i$th exterior power of the cotangent complex of $BG$, isomorphic to $S^{i}(\fg^{*})[-i]$, the $i$-th symmetric power of the dual to the adjoint representation of $G$, placed in the cohomological degree $i$. 

\begin{rem}
\begin{enumerate}
\item The Hodge cohomology groups $H^j(BG, \Lambda^i L_{BG/k})$ of $BG$ can be identified with the cohomology $H^{j-i}(G, S^j(\mathfrak{g}^*))$ of the group $G$ with coefficients in the symmetric powers of the dual to the adjoint representation, cf. \cite[Corollary 2.2]{totaro}.
\item The fact that there are no non-zero differentials coming out of terms $E_r^{i,j}$ in the range $i+j<p$ was proven in \cite[Corollary 1.3.24]{kubrak-prikhodko-modp} by a generalization to stacks of the method of Deligne-Illusie. Our proof also proceeds by proving first that the conjugate spectral sequence degenerates for $BG$.
\item We do not know if Theorem \ref{stacks: main htdr bg} remains true over an arbitrary base ring $k$. Degeneration trivially holds when $k$ is a $\bQ$-algebra because in this case the entries $E_1^{i,j}$ of the spectral sequence are zero for $i\neq j$. The case of $k=\bZ/p^n$ for $n\geq 2$ seems to require a new idea.
\item We find it somewhat curious that replacing a reductive group $G$ by a quotient $G'=G/\Gamma$ by a central finite subgroup scheme $\Gamma\subset Z(G)$ of multiplicative type (e.g. $G=GL_n$ and $G'=PGL_n$) might make the Hodge cohomology of the classifying stack much more complicated but the Frobenius-splitting of $BG'$ easily follows (Lemma \ref{stacks: isogeny invariant}) from that of $BG$, so our proof of Theorem \ref{intro: classifying stack} for $G'$ is no more harder than it is for $G$.
\end{enumerate}
\end{rem}

We refer the reader to \cite[1.1]{kubrak-prikhodko-modp} or \cite[Construction 2.7]{antieau-bhatt-mathew} for the definition and basic properties of Hodge and de Rham cohomology of smooth stacks. Following the same principle as in these references, we introduce the de Rham complex of a smooth stack $X$ in characteristic $p$, viewed as an object of the derived category of quasi-coherent sheaves on $X^{(1)}$, by formally extending it from the corresponding invariant for smooth schemes.

For the duration of this section, $k$ is a perfect $\bF_p$-algebra. The definition of the de Rham complex for a stack will be a special case of the following construction that, given a functorial way of assigning to every smooth $k$-scheme $S$ an object $\cF_S\in D(S)$, extends it to arbitrary smooth Artin stacks. To simplify the exposition, we restrict ourselves to stacks with affine diagonal.

\renewcommand{\cP}{\Sm}
\begin{constr}
Following \cite[020S]{kerodon}, we denote by $\QC_{\obj}$ the $\infty$-category of $\infty$-categories equipped with a distinguished object. Its objects are pairs $(\cC,C)$ where $\cC$ is an $\infty$-category, and $C\in \cC$ is an object, and $1$-morphisms from $(\cC,C)$ to $(\cD, D)$ are pairs $F:\cC\to \cD, \alpha:F(C)\to D$ where $F$ is a functor and $\alpha$ is any $1$-morphism in $\cD$, which does not have to be invertible. Likewise, $\QC$ denotes the $\infty$-category of $\infty$-categories \cite[0208]{kerodon}.

Let $\cP$ be the category of smooth affine schemes over $k$. Denote by $\St_{\cP}$ the category of smooth Artin stacks over $k$ with affine diagonal, in the sense of \cite[026O]{stacks}. 

Suppose we are given a functor $\cF:\cP^{\op}\to \QC_{\obj}$ such that the underlying functor $\cP^{\op}\xrightarrow{\cF} \QC_{\obj}\to \QC$ is the functor $D(-)$, sending $S\in \Sm$ to the derived category of quasi-coherent sheaves on $S$, and a morphism $f:S\to S'$ to the pullback functor $f^*:D(S')\to D(S)$. For an affine scheme $S\in\cP$ we denote the distinguished object of $D(S)$ provided by $\cF$ by $\cF_S\in D(S)$. Let $\tcF:\St_{\cP}^{\op}\to \QC_{\obj}$ be the left Kan extension of $\cF$ along $\cP^{\op}\hookrightarrow \St_{\cP}^{\op}$. Explicitly, for a stack $X\in \St_{\cP}^{\op}$ we have $\tcF(X)=(D(X),\tcF_X)$ where the object $\tcF_X\in D(X)$ is given by 
\begin{equation}
\tcF_X=\lim\limits_{f:S\to X}f_*\cF_S
\end{equation}
where the limit is taken over all smooth affine schemes mapping to $X$. By the assumption that $X$ has affine diagonal, each morphism $f$ is necessarily affine.
\end{constr}

\begin{lm}\label{stacks: descent of kan extension}
If the functor $\cF:\cP^{\op}\to \QC_{\obj}$ satisfies smooth descent, then so does the extended functor $\tcF$. 
\end{lm}
\begin{proof}
Let $\pi:X'\to X$ be a smooth cover in $\St_{\cP}$. If $\pi':U\to X'$ is a smooth cover by a scheme, then it suffices to check descent along morphisms $\pi'$ and $\pi\circ \pi'$ both of which are representable, because $X'$ and $X$ are assumed to have affine diagonal. Hence it suffices to treat the case of a representable smooth cover $\pi$.

The \v{C}ech nerve of $\pi$ consists of stacks $\pi_n:X'^{\times_X (n+1)}\to X$. We need to show that the natural map
\begin{equation}\label{stacks: descent map formula}
\tcF_{X}\to\lim\limits_{[n]\in\Delta} \pi_{n*}\tcF_{X'^{\times_X (n+1)}} 
\end{equation}
is an equivalence.

Consider $1$-categories $\cC$ and $\cC'$ defined as follows. An object of $\cC$ consists of a pair $([n],f_n:T\to X'^{\times_X (n+1)})$ where $[n]$ is an object of $\Delta^{\op}$, and $f_n$ is an arbitrary morphism from a smooth affine scheme $T\in \cP$. The other category $\cC'$ consists of pairs $([n],f:S\to X)$ where $[n]\in \Delta^{\op}$, and $S$ is a smooth affine scheme. There is a functor $\cC'\to \cC$ sending $([n], S\to X)$ to $([n], S\times_X X'^{\times_X (n+1)}\to X'^{\times_X (n+1)})$.

We have a functor $\cF_{\pi}:\cC\to D(X)$ sending $([n],f_n:T\to X'^{\times_X (n+1)})$ to $\pi_{n*}f_{n*}\cF_T$. By definition, the codomain of (\ref{stacks: descent map formula}) is computed as $\lim\limits_{\cC^{\op}}\cF_{\pi}$. For each morphism $f:S\to X$ from a smooth affine scheme $S\in \cP$ the fiber product $f:S':=S\times_X X'\to S$ is a smooth affine cover of $S$. By smooth descent for $\cF$, we can then calculate $\tcF_X$ as $\lim\limits_{f:S\to X}\lim\limits_{[n]\in\Delta} f_*\pi_{n,S*}\cF_{S'^{\times_S (n+1)}}$. In other words, $\tcF_{X}$ is equivalent to the limit $\lim\limits_{\cC'^{\op}}\cF_{\pi}|_{\cC'}$. 

The functor $\cC'\to \cC$ is cofinal by Quillen's theorem A \cite[Theorem 4.1.3.1]{lurie-htt}, because for an object $([n],f_n:T\to X'^{\times_X (n+1)})\in\cC$ the slice category $\cC'\times_{\cC}\cC_{([n], T)/}$ has the initial object $S=T\xrightarrow{f_n}X'^{\times_X (n+1)}\xrightarrow{\pi_n}X$ and is therefore weakly contractible. Hence limits of $\cF_{\pi}$ over categories $\cC^{\op}$ and $\cC'^{\op}$ are equivalent, as desired.
\end{proof}
\renewcommand{\cP}{\mathcal{P}}

Specializing the above construction to functors $S\mapsto F_{S/k*}\Omega^{\bullet}_{S/k}\in D(S^{(1)})$ and $S\mapsto \tau^{\leq n}F_{S/k*}\Omega^{\bullet}_{S/k}\in D(S^{(1)})$ we can define the de Rham complex of a smooth Artin stack, together with its conjugate filtration:

\begin{definition}
For a smooth Artin stack $X$  with affine diagonal over a perfect $\bF_p$-algebra $k$ define the following object of $D(X^{(1)})$:
\begin{equation}
F_*\Omega^{\bullet}_{X}:=\lim\limits_{f:S\to X}f^{(1)}_*F_{S/k*}\Omega^{\bullet}_{S/k}
\end{equation}
where the limit is taken over all morphisms $f:S\to X$ from affine smooth schemes to $X$. Similarly, for $n\geq 0$ define the $n$-th step of the conjugate filtration on the de Rham complex as
\begin{equation}
\Fil_n^{\conj}F_*\Omega^{\bullet}_{X}:=\lim\limits_{f:S\to X}f^{(1)}_*\tau^{\leq n}F_{S/k*}\Omega^{\bullet}_{S/k}
\end{equation}
where the limit is taken over the same category of affine schemes that are smooth over $X$.
\end{definition}

Note that $F_{*}\Omega^{\bullet}_{X}$ is an indivisible piece of notation here: we did not define the de Rham complex itself as a sheaf-theoretic object on $X$.

\begin{rem}
A more finitary method of computing $F_{*}\Omega^{\bullet}_{X/k}$ is given by the smooth descent for it, as in Lemma \ref{stacks: basic properties} (4) below. For a smooth scheme $X$ over $k$ the above definition correctly recovers the Frobenius pushforward of the de Rham complex -- for an affine $X$ the diagram over which the limit is taken has a final object given by $S=X$, and the case of a general scheme $X$ follows by descent.
\end{rem}

For a stack $Y$ over $k$ we denote by $\RGamma(Y,-):D(Y)\to D(k)$ the derived pushforward functor along the morphism $Y\to \Spec k$.

\begin{lm}\label{stacks: basic properties}For a smooth Artin stack $X$ with affine diagonal over $k$ the following holds.
\begin{enumerate}
\item The complex of $k$-modules $\RGamma(X^{(1)},F_{*}\Omega^{\bullet}_{X/k})$ is naturally equivalent to the de Rham cohomology $\RGamma_{\dR}(X/k)$ of the stack $X$, as defined e.g. in \cite[Definition 1.1.3]{kubrak-prikhodko-modp}.
\item Conjugate filtration on $F_{*}\Omega^{\bullet}_{X/k}$ is exhaustive, that is the natural map 
\begin{equation}
\colim\limits_n\Fil_n^{\conj}F_{*}\Omega^{\bullet}_{X}\to F_{*}\Omega^{\bullet}_{X}
\end{equation}
is an equivalence.
\item The $n$-th graded piece $\gr_n^{\conj}:=\cone(\Fil_{n-1}F_{*}\Omega^{\bullet}_{X}\to \Fil_n F_{*}\Omega^{\bullet}_{X})$ of the conjugate filtration is equivalent to $\Lambda^n L_{X^{(1)}/k}[-n]$ where $L_{X^{(1)}/k}$ is the cotangent complex of the stack $X^{(1)}$.
\item For a smooth surjective morphism $f:U\to X$ of smooth Artin stacks with affine diagonals there is a natural equivalence
\begin{equation}\label{stacks: de rham complex descent formula}
\Fil_n^{\conj}F_{*}\Omega^{\bullet}_{X}\simeq \lim\limits_{[i]\in \Delta} f_{i*}\Fil^n_{\conj}F_{*}\Omega^{\bullet}_{U^{\times_X (i+1)}}
\end{equation}
where $f_i:U^{\times_X(i+1)}\to X$ are structure maps from the terms of the Cech nerve of $f$. 
\end{enumerate}
\end{lm}
\begin{proof}
(1) follows from the natural equivalence $\RGamma(S^{(1)},F_{*}\Omega^{\bullet}_{S})\simeq\RGamma_{\dR}(S/k)$ for smooth affine schemes $S$ and the fact that $\RGamma(X^{(1)},-)$ commutes with limits.

Statement (2) is equivalent to proving that $\colim\limits_n\lim\limits_{f:S\to X}f^{(1)}_*\tau^{>n}F_{*}\Omega^{\bullet}_{S}$ vanishes. Since $f^{(1)}_*$ is left exact with respect to the standard $t$-structures on $D(S^{(1)})$ and $D(X^{(1)})$, each object $f^{(1)}_*\tau^{>n}F_{*}\Omega^{\bullet}_{S}$ is concentrated in degrees $>n$. The subcategory of objects concentrated in cohomological degrees $>n$ is closed under limits \cite[Corollary 1.2.1.6]{lurie-ha}. Therefore for each $n$ the object $\lim\limits_{f:S\to X}f^{(1)}_*\tau^{>n}F_{*}\Omega^{\bullet}_{S}\in D(X^{(1)})$ is concentrated in degrees $>n$, and their colimit along $n$ vanishes, finishing the proof of (2).

By Cartier isomorphism, for a smooth scheme $S$ the $n$-th graded piece of the conjugate filtration on $F_{*}\Omega^{\bullet}_{S}$ is naturally equivalent to $\Omega^n_{S^{(1)}/k}[-n]$. Hence $\gr_n^{\conj}F_{*}\Omega^{\bullet}_{X}$ is computed by the limit $\lim\limits_{S\to X}f^{(1)}_*\Omega^n_{S^{(1)}/k}[-n]$, which coincides with the shifted $n$-th exterior power of the cotangent complex $L_{X^{(1)}/k}$, as a consequence of smooth descent for cotangent complex, cf. \cite[Proposition 1.1.4]{kubrak-prikhodko-modp} or \cite[Remark 2.8]{antieau-bhatt-mathew}. This proves statement (3).

By smooth descent for cotangent complex and its exterior powers \cite[Remark 2.8]{bhatt-completions}, the assignment $S\mapsto \tau^{\leq n}F_{*}\Omega^{\bullet}_{S}$ satisfies smooth descent, so (4) follows by Lemma \ref{stacks: descent of kan extension}.
\end{proof}

\begin{example}
Explicitly, for a smooth affine group scheme $G$ over $k$ the object $F_*\Omega^{\bullet}_{BG/k}\in D^+(BG^{(1)})\simeq D^+(\Rep_k G^{(1)})$ is a complex of $k$-linear representations of $G^{(1)}$ that can be represented by the totalization of the cosimplicial complex
\begin{equation}\label{stacks: bg example}
\begin{tikzcd}
F_{G/k*}\Omega^{\bullet}_{G/k}\arrow[r, shift left=0.65ex] \arrow[r, shift right=0.65ex] & F_{G\times G/k*}\Omega^{\bullet}_{G\times G/k} \arrow[r, shift left=1.3ex] \arrow[r, shift right=1.3ex] \arrow[r] &\ldots
\end{tikzcd}
\end{equation}
obtained by applying the functor of Frobenius linearized de Rham complex to the \v{C}ech nerve of the map $G\to \Spec k$. For each $n\geq 1$, the group scheme $G$ acts on $\Omega^{\bullet}_{G^{\times n}/k}$ via diagonal translations, and consequently $G^{(1)}$ acts on $F_{G^{\times n}*}\Omega^{\bullet}_{G^{\times n}/k}$, giving an action of $G^{(1)}$ on the cosimplicial complex of $k$-modules (\ref{stacks: bg example}).
\end{example}

\begin{pr}\label{stacks: frobenius decomposition}
For a smooth Artin stack $X$  with affine diagonal over a perfect $\bF_p$-algebra $k$ there is an equivalence
\begin{equation}\label{stacks: frobenius decomposition formula}
F_{X/k}^*F_{*}\Omega^{\bullet}_{X/k}\simeq \bigoplus\limits_{i\geq 0}F^*_{X/k} \Lambda^i L_{X^{(1)}/k}[-i]
\end{equation}
\end{pr}
\begin{proof}
Given Proposition \ref{dr stack: decomposition after frobenius}, this follows formally from smooth descent for Frobenius pullback of the cotangent complex, established in Lemma \ref{stacks: frobenius pullback descent} below.
\end{proof}

For a future application, we prove the following in a larger generality of not necessarily smooth stacks. The Frobenius endomorphism $F_X$ for a non-smooth stack $X$ might not be flat, and everywhere below for a morphism $f:X\to Y$ we denote by $f^*:D(Y)\to D(X)$ the {\it derived} pullback functor.

\begin{lm}\label{stacks: frobenius pullback descent}
Let $f:U\to X$ be a flat surjective morphism of arbitrary Artin stacks over $\bF_p$. Consider the simplicial \v{C}ech nerve $U_{\bullet}:=U^{\times_X\bullet}$, each term equipped with the structure map $f_i:U^{\times_X (i+1)}\to X$. For all $n$ the natural map \begin{equation}\label{stacks: frobenius pullback descent map}F_{X}^*\Lambda^n L_{X}\to \lim\limits_{\Delta} f_{\bullet *}F_{U_{\bullet}}^*\Lambda^n L_{U_{\bullet}}\end{equation} is an equivalence.

In particular for a smooth Artin stack $X$ with affine diagonal over a perfect $\bF_p$-algebra $k$, the natural map $F_{X}^*\Lambda^nL_{X}\to \lim\limits_{f:S\to X}f_*F_{S}^*\Lambda^n L_{S}$,  where the limit is taken over all smooth affine schemes $S$ mapping to $X$, is an equivalence.
\end{lm}
\begin{proof}
The proof follows the same idea as Bhatt's proof \cite[Theorem 3.1]{bms2} of descent for cotangent complex itself. First of all, as in the proof of Lemma \ref{stacks: descent of kan extension}, for the first assertion of the lemma it suffices to treat the case of a representable $f$. By flat descent for quasi-coherent sheaves, $F_{X}^*\Lambda^nL_{X}$ is equivalent to the limit $\lim\limits_{\Delta} f_{\bullet *}f_{\bullet}^*F_{X}^*\Lambda^nL_{X}\simeq \lim\limits_{\Delta} f_{\bullet *}F_{U_{\bullet}}^*f_{\bullet}^*\Lambda^nL_{X}$. Hence the fact that (\ref{stacks: frobenius pullback descent map}) is an equivalence in $D(X)$ is equivalent to the vanishing of the object
\begin{equation}
\lim\limits_{\Delta} f_{\bullet*}F^*_{U_{\bullet}}\cone(f_{\bullet}^*\Lambda^nL_{X}\xrightarrow{df_{\bullet}} \Lambda^n L_{U_{\bullet}})\in D(X).
\end{equation}
Using the fundamental triangle $f_{\bullet}^*L_{X}\to L_{U_{\bullet}}\to L_{U_{\bullet}/X}$, we can endow the cone $\cone(df_{\bullet}:f_{\bullet}^*\Lambda^nL_{X}\to \Lambda^n L_{U_{\bullet}})$ with a finite filtration whose graded pieces are $\Lambda^{j}L_{U_{\bullet}/X}\otimes f_{\bullet}^*\Lambda^{n-j}L_X$ with $0<j\leq n$.

Hence it suffices to check that the limit $\lim\limits_{\Delta} f_{\bullet *}F_{U_{\bullet}}^*(\Lambda^j L_{U_{\bullet}/X}\otimes f_{\bullet}^*\Lambda^{n-j} L_X)$ vanishes for all $0<j\leq n$. This can be checked flat locally on $X$ and \'etale locally on $U$, so we may and do assume that $X=\Spec A$ and $U=\Spec B$ are both affine schemes. 

We need to check that the totalization of the following cosimplicial object of $D(A)$ \begin{equation}\label{stacks: affine cotangent totalization to kill}
\Delta\ni [i]\mapsto F^*_{B^{\otimes_A (i+1)}}(\Lambda^j L_{B^{\otimes_A (i+1)}/A})\otimes_A F_A^*\Lambda^{n-j}L_{A}\end{equation} is zero. Since $A\to B$ if a faithfully flat map, this can be checked after tensoring this cosimplicial $A$-module with $B$. We will now prove that the cosimplicial object $[i]\to F^*_{B^{\otimes_A (i+1)}}(\Lambda^j L_{B^{\otimes_A (i+1)}/A})\otimes_A B$ of $D(B)$ is contractible. This implies the desired vanishing of the totalization of (\ref{stacks: affine cotangent totalization to kill}) because this cosimplicial object remains contractible after tensoring with the complex $F_A^*\Lambda^{n-j}L_{A}$ over $A$.

By base change for the cotangent complex, for an $A$-algebra $C$ there is an equivalence $F_C^*\Lambda^n L_{C/A}\otimes_A B\simeq F_{C\otimes_A B}^*\Lambda^n L_{(C\otimes_A B)/B}$ natural in $C$. In particular, the cosimplicial object $[i]\mapsto F^*_{B^{\otimes_A (i+1)}}\Lambda^j L_{B^{\otimes_A (i+1)}/A}\otimes_A B$ of $D(B)$ is equivalent to the result of applying the functor $D\mapsto F^*_D\Lambda^n L_{D/B}$ from $B$-algebras to complexes of $B$-modules to the \v{C}ech nerve of the map $\id_B\otimes 1:B\to B\otimes_A B$. This \v{C}ech nerve is homotopy equivalent to the constant cosimplicial $B$-algebra $B$, hence the result of applying functor $D\mapsto F_D^*\Lambda^n L_{D/B}$ is homotopy equivalent to the constant cosimplicial object with value $F_B^*\Lambda^n L_{B/B}=0$, as desired.

The second assertion of the lemma follows formally from the descent statement that we just established: the map in question is an equivalence when $X$ is a smooth scheme, and by Lemma \ref{stacks: descent of kan extension} its codomain satisfies descent as well, hence it is an equivalence for all smooth Artin stacks. 
\end{proof}

We introduce a generalization of the notion of Frobenius splitting to Artin stacks:

\begin{definition}
For an Artin stack $X$ over $\bF_p$ a Frobenius splitting is a map $\tau:RF_{X*}\cO_X\to \cO_X$ in the derived category of quasi-coherent sheaves on $X$, such that composition $\cO_X\xrightarrow{F_X^{\#}}RF_{X*}\cO_X\xrightarrow{\tau} \cO_X$ equals the identity.
\end{definition}

For a scheme $X$ the Frobenius endomorphism is affine, so the derived pushforward $RF_{X*}\cO_X$ is identified with the plain pushforward $F_{X*}\cO_X$, but e.g. for the classifying stack $X=BG$ of a non-commutative connected reductive group the derived pushforward $RF_{BG*}\cO_{BG}$ is concentrated in infinitely many cohomological degrees \cite[Proposition II.12.13]{jantzen}.

\begin{pr}\label{stacks: bg is fsplit}
For a connected reductive group $G$ over a perfect field $k$ the classifying stack $BG$ admits a Frobenius splitting.
\end{pr}

\begin{proof}
We will deduce this from properties of the Steinberg representation of $G$, following the idea of the proof of \cite[Theorem 2.1]{cline-parshall-scott}. We denote by $G^{(1)}$ the Frobenius twist of $G$, viewed as a group scheme over $k$. Note that its clasiffying stack $BG^{(1)}$ is the Frobenius twist of the $k$-stack $BG$, and the Frobenius morphism $F_{BG}:BG\to BG^{(1)}$ is induced by the relative Frobenius map $F_G:G\to G^{(1)}$ of the $k$-group scheme $G$.

Recall that the bounded below derived category $D^+(BG^{(1)})$ is equivalent to the derived category $D^+(\Rep G^{(1)})$ of the abelian category of algebraic representations of $G^{(1)}$. Under this equivalence the structure sheaf $\cO_{BG^{(1)}}$ gets carried to the trivial $1$-dimensional representation of $G^{(1)}$.

Denote by $G_1:=\ker(G\xrightarrow{F_{G}} G^{(1)})$ the Frobenius kernel of $G$, it is a finite local group scheme over $k$. In terms of representations, the functor $RF_{BG*}:D^+(BG)\to D^+(BG^{(1)})$ is described as sending $V\in D^+(BG)\simeq D^+(\Rep G)$ to the cohomology complex $\RGamma(G_1,V|_{G_1})$ of the Frobenius kernel, equipped with the residual action of $G^{(1)}\simeq G/G_1$. The statement of the proposition is then equivalent to proving that the map $k\simeq H^0(G_1,k)\to \RGamma(G_1, k)$ in $D(\Rep G^{(1)})$ has a splitting.

For an extension $k\subset k'$ of perfect fields the natural map $\RGamma(G_1,k)\otimes_k k'\to \RGamma((G_{k'})_1,k')$ is an equivalence, so for the purposes of proving the proposition we may assume that $k$ is algebraically closed. In particular, we may choose a Borel subgroup $B\subset G$ with a split maximal torus. Furthermore, we may assume that $G$ is semi-simple and simply connected because our statement is invariant under replacing $G$ by an isogenous reductive group:

\begin{lm}\label{stacks: isogeny invariant}
Let $\psi:G\to G'$ be a surjection of reductive groups over $k$ such that $\ker\psi$ is a group scheme of multiplicative type (i.e. embeddable into $\bG_{m,k}^r$) contained in the center of $G$. The map $k\to \RGamma(G_1,k)$ has a splitting in $D^+(\Rep G^{(1)})$ if and only if the map $k\to \RGamma(G'_1,k)$ has a splitting in $D^+(\Rep G^{'(1)})$.
\end{lm}

\begin{proof}
We have a short exact sequence of finite group schemes $1\to (\ker \psi)_1\to G_1\to G'_1\to 1$, and since the group scheme $(\ker\psi)_1$ has no higher cohomology, the natural map $\RGamma(G'_1,k)\to \RGamma(G_1,k)$ is a quasi-isomorphism. This equivalence is compatible with the actions of $G^{(1)}$ and $G^{'(1)}$: the complex $\RGamma(G_1,k)$ is the object of $D^+(\Rep G^{(1)})$ obtained from $\RGamma(G'_1,k)\in D^+(\Rep G^{'(1)})$ by restricting along the map $\psi^{(1)}:G^{(1)}\to G^{'(1)}$. 

Therefore the class in $\Hom_{D(G^{(1)})}(\tau^{\geq 1}\RGamma(G_1,k),k[2])$ obstructing the splitting of $k\to \RGamma(G_1,k)$ is the image of the analogous obstruction for $G'$ under the map \begin{equation}
    \Hom_{D(G^{'(1)})}(\tau^{\geq 1}\RGamma(G_1,k),k[2])\to \Hom_{D(G^{(1)})}(\tau^{\geq 1}\RGamma(G_1,k),k[2])
\end{equation}
which is injective because $\ker\psi^{(1)}$ has no higher cohomology. Hence one of these two obstructions vanishes if and only if the other does.
\end{proof}

The assumption that $G$ is simply connected allows us to consider the Steinberg representation $\St$ of $G$ defined as the irreducible representation of highest weight $(p-1)\rho$, where $\rho$ is the weight of the maximal torus $T\subset B$ equal to the half sum of all positive roots. For example, for $G=SL_2$ the Steinberg representation is the $(p-1)$th symmetric power $\Sym^{p-1}V$ of the $2$-dimensional tautological representation $V$.

The restriction of $\St$ to the Frobenius kernel $G_1$ is an irreducible projective object of the category of representations of $G_1$ by \cite[II.3.10(3), II.3.18, Proposition II.10.2]{jantzen}. In particular, $Q:=\Hom_k(\St,\St)\simeq \St\otimes\St^{\vee}$ is a projective object, and since every finite-dimensional projective representation of a group scheme is also injective, $H^i(G_1,Q)$ vanishes for $i>0$, while $H^0(G_1,Q)$ is $1$-dimensional, spanned by $\id_{\St}\in Q$.

Therefore the map $k\xrightarrow{\id_{\St}}Q$ gives rise to a map $\RGamma(G_1, k)\to \RGamma(G_1,Q|_{G_1})\simeq k$ in $D(\Rep G^{(1)})$ inducing an isomorphism on $0$th cohomology, and this is the desired splitting.
\end{proof}

We can now deduce Theorem \ref{stacks: main htdr bg}, just as in Corollary \ref{qfsplit: qfsplit decomposable}.

\begin{proof}[Proof of Theorem \ref{stacks: main htdr bg}]
Let us check that the morphism $F_{BG}:BG\to BG^{(1)}$ satisfies projection formula on the derived category of bounded below objects $D^+(BG^{(1)})\simeq D^+(\Rep G^{(1)})$. For a complex of representations $V\in D^+(\Rep G^{(1)})$ the natural map $V\otimes_{\cO_{BG}}RF_{BG*}\cO_{BG}\to RF_{BG*}F_{BG}^*V$ is described in terms of group cohomology as the map $V\otimes_k \RGamma(G_1, k)\to \RGamma(G_1, V^{\triv})$, where $V^{\triv}$ is the complex of $k$-vector spaces underlying $V$ equipped with the trivial $G_1$-action. This map is an equivalence because $k[G_1]$ is finite-dimensional over $k$, so the formation of the standard bar complex of a $G_1$-representation $W$ commutes with arbitrary colimits in $W$.

Therefore Proposition \ref{stacks: frobenius decomposition} applied to $X=BG$ can be restated as
\begin{equation}
RF_{BG*}\cO_{BG}\otimes_{\cO_{BG^{(1)}}}F_{*}\Omega^{\bullet}_{BG/k}\simeq \bigoplus\limits_{i\geq 0}RF_{BG*}\cO_{BG}\otimes_{\cO_{BG^{(1)}}} \Lambda^iL_{BG^{(1)}/k}[-i]
\end{equation}
As in the proof of Corollary \ref{qfsplit: qfsplit decomposable}, the fact that the map $\cO_{BG^{(1)}}\to RF_{BG*}\cO_{BG}$ admits a splitting gives a quasi-isomorphism
\begin{equation}\label{stacks: bg de rham complex decomposition formula}
F_{*}\Omega^{\bullet}_{BG/k}\simeq \bigoplus\limits_{i\geq 0} \Lambda^iL_{BG^{(1)}/k}[-i]
\end{equation}
We can now finish the proof via the Deligne-Illusie argument using finite-dimensionality of Hodge cohomology for $BG$, cf. \cite[Corollary 1.3.24]{kubrak-prikhodko-modp}. Decomposition (\ref{stacks: bg de rham complex decomposition formula}) implies that the conjugate spectral sequence $E_{2}^{i,j}=H^i(BG^{(1)},\Lambda^j L_{BG^{(1)}/k})\Rightarrow H^{i+j}_{\dR}(BG/k)$ for the stack $BG$ degenerates at the second page. For all $i,j$ the term $E_{2}^{i,j}\simeq H^{i-j}(G,S^j(\mathfrak{g}^*))^{(1)}$ is a finite-dimensional $k$-vector space by \cite[Proposition II.4.10(a) + Corollary II.4.7(c)]{jantzen}. Hence degeneration of the conjugate spectral sequence gives equality of dimensions $\dim_k H^{n}_{\dR}(BG/k)=\sum\limits_{i+j=n}\dim_k H^i(BG^{(1)},\Lambda^j L_{BG^{(1)}/k})$, and the Hodge-to-de Rham spectral sequence degenerates by dimension reasons.
\end{proof}

\section{Decomposition after Frobenius pullback via semi-perfect descent}\label{qrsp}

In this section we give another proof of the decomposition (\ref{intro: frob decomposition formula}), in the case $S=\Spec k$ is the spectrum of a perfect field, which does not use the de Rham stack, but rather appeals to computing de Rham cohomology of a smooth scheme in characteristic $p$ via descent from quasiregular semi-perfect algebras.

\begin{definition}[{\hspace{1sp}\cite[Definition 8.8]{bms2}}]
An $\bF_p$-algebra $R$ is {\it quasiregular semiperfect} if the Frobenius $\varphi_R:R\to R$ is surjective, and the cotangent complex $L_{R/\bF_p}$ is concentrated in cohomological degree $-1$, with $H^{-1}(L_{R/\bF_p})$ being a flat $R$-module.
\end{definition}

For $R$ quasiregular semiperfect we denote by $R^{\flat}:=\lim\limits_{\varphi_R}R$ the inverse limit perfection of $R$. It surjects onto $R$, and we denote the kernel of the surjection by $I:=\ker(R^{\flat}\twoheadrightarrow R)$. The derived de Rham cohomology of $R$ relative to $\bF_p$ is concentrated in degree $0$ and can be described, together with its conjugate filtration, as follows.

Let $D_I(R^{\flat})$ denote the divided power envelope of the ideal $I\subset R^{\flat}$. By its universal property, there is a surjection $D_I(R^{\flat})\to R$, and the natural map $R^{\flat}\to D_I(R^{\flat})$ factors through $R^{\flat}/\varphi_{R^{\flat}}(I)$, making $D_I(R^{\flat})$ into a $R^{\flat}/\varphi_{R^{\flat}}(I)$-algebra. Since $R^{\flat}$ is perfect, its Frobenius endomorphism gives rise to an isomorphism $R=R^{\flat}/I\simeq R^{\flat}/\varphi_{R^{\flat}}(I)$ and we view $D_I(R^{\flat})$ as an $R$-algebra via this isomorphism.

Following \cite[Section 8]{bms2}, we define the $\bN$-indexed increasing conjugate filtration on $D_I(R^{\flat})$. Let $\Fil_n^{\conj}D_I(R^{\flat})$ be the $R^{\flat}$-submodule of $D_I(R^{\flat})$ generated by elements of the form $a_1^{[l_1]}\cdot\ldots \cdot a_m^{[l_m]}$ with $a_1,\ldots, a_m\in I$ and $l_1,\ldots l_m\in\bN$ such that $\sum\limits l_i<(n+1)p$. We denote by $\dR_{R/\bF_p}$ the derived de Rham cohomology of $R$, equipped with the conjugate filtration, cf. \cite[Proposition 3.5]{bhatt-derived-derham}. The $0$th step of the conjugate filtration is $R$ itself, which makes $\dR_{R/\bF_p}$ into a (a priori derived) $R$-algebra.

\begin{pr}[{\hspace{1sp}\cite[Proposition 8.12]{bms2}}]
There is a natural isomorphism $\dR_{R/\bF_p}\simeq D_I(R^{\flat})$ of discrete $R$-algebras, carrying the conjugate filtration on the derived de Rham cohomology to the filtration $\Fil^{\conj}_{\bullet}D_I(R^{\flat})$ defined above.
\end{pr}

Recall also that associated graded algebra of the conjugate filtration $\Fil_{\bullet}^{\conj}D_{I}(R^{\flat})$ is identified with the free divided power algebra generated by the flat $R$-module $I/I^2$, via the map\footnote{In \cite[Proposition 8.11(3)]{bms2} one finds a different formula, but it simplifies modulo $p$ because $\frac{(pk)!}{p^kk!}$ equals $(-1)^k$ modulo $p$, for all $k\geq 1$.}
\begin{multline}\label{qrsp: graded pieces map}
a_1^{[l_1]}\cdot\ldots\cdot a_m^{[l_m]}\mapsto (-1)^{l_1+\ldots+l_m}\ta_1^{[pl_1]}\cdot\ldots\cdot \ta_m^{[pl_m]}:\Gamma^n_{R}(I/I^2)\to \gr_n^{\conj}D_I(R^{\flat})
\end{multline}
where $\ta_i\in I\subset R^{\flat}$ are arbitrary lifts of $a_i\in I/I^2$; the image of the element $\ta_1^{[pl_1]}\cdot\ldots\cdot \ta_m^{[pl_m]}$ under the map $\Fil_n^{\conj}D_I(R^{\flat})\to \gr_n^{\conj}D_{I}(R^{\flat})$ does not depend on the choice of the lifts. 

We prove an analog of Proposition \ref{dr stack: decomposition after frobenius} for quasiregular semiperfect rings, which becomes a construction in plain commutative algebra because the de Rham complex is now concentrated in a single degree:

\begin{pr}\label{qrsp: affine decomposition}
There is a natural isomorphism of filtered $R$-algebras
\begin{equation}
s:\bigoplus\limits_{n\geq 0}\varphi_R^*\Gamma^n_R(I/I^2)\simeq \varphi_R^*D_I(R^{\flat})
\end{equation}
identifying the base change of the conjugate filtration on the right-hand side with the filtration induced by the grading on the left-hand side.
\end{pr}

\begin{proof}
Since $\varphi_{R}$ is surjective, we can identify the base change $\varphi_R^*D_I(R^{\flat})=D_I(R^{\flat})\otimes_{R,\varphi_{R}}R$ with the quotient $D_I(R^{\flat})/(\ker \varphi_{R}\cdot D_I(R^{\flat}))$. Define the splitting map $s$ on an element $a_1^{[l_1]}\cdot\ldots\cdot a_m^{[l_m]}\otimes 1\in\Gamma^n_R(I/I^2)\otimes_{R,\varphi_R}R$ by
\begin{equation}\label{qrsp: splitting formula}
s(a_1^{[l_1]}\cdot\ldots\cdot a_m^{[l_m]}\otimes 1):=(-1)^{l_1+\ldots+l_m}\ta_1^{[pl_1]}\cdot\ldots\cdot \ta_m^{[pl_m]}\in \Fil_n^{\conj}D_I(R^{\flat})/\ker\varphi_{R}
\end{equation}
where $\ta_i\in I\subset R^{\flat}$ are arbitrary lifts of $a_i\in I/I^2$. Let us check that the value of (\ref{qrsp: splitting formula}) does not depend on the choice of these lifts. Recall that an element $r\in R$ acts on $D_I(R^{\flat})$ via multiplication by $\widetilde{r}^p\in R^{\flat}$ where $\widetilde{r}\in R^{\flat}$ is any lift of $r$. In particular, $\ker\varphi_R\cdot \Fil_n^{\conj}D_I(R^{\flat})$ is simply $I\cdot \Fil_n^{\conj}D_I(R^{\flat})$. 

For elements $a\in I, b\in I^2$ we have $(a+b)^{[l]}=a^{[l]}+\sum\limits_{i>0}a^{[l-i]}\cdot b^{[i]}$ for any $l\geq 0$. For $i>0$ the divided power $b^{[i]}\in D_I(R^{\flat})$ lies in $\ker\varphi_{R}\cdot D_I(R^{\flat})$, because $b$ is a sum of elements of the form $b_1\cdot b_2$ with $b_1,b_2\in I$ and the element $(b_1b_2)^{[i]}=b_1^i\cdot b_2^{[i]}$ belongs to $I\cdot D_I(R^{\flat})$. Therefore, for a given $a\in I$ the element $(a+b)^{[l]}$ in $D_I(R^{\flat})/\ker\varphi_{R}\cdot D_I(R^{\flat})$ does not depend on the choice of $b\in I^2$.

The image of $\ta_1^{[pl_1]}\cdot\ldots\cdot \ta_m^{[pl_m]}$ under the map $$\Fil_n^{\conj}D_I(R^{\flat})/\ker\varphi_{R}\to \gr_n^{\conj}D_I(R^{\flat})/\ker\varphi_{R}$$ goes to $a_1^{[l_1]}\cdot\ldots\cdot a_m^{[l_m]}$ under the map (\ref{qrsp: graded pieces map}), hence $s$ extends to a map of $R$-modules, which is necessarily an isomorphism, as desired.
\end{proof}

We can now deduce the corresponding statement for smooth varieties over $k$. 

\begin{proof}[{Proof of (\ref{intro: frob decomposition formula}) from Theorem \ref{intro: main}}]For a smooth scheme $X$ over $k$ the perfection $X_{\perf}:=\lim\limits_{F_X}X$ is a faithfully flat cover of $X$, and the \v{C}ech nerve of this cover
\begin{equation}
\begin{tikzcd}
X_{\perf}\arrow[r, shift left=0.65ex] \arrow[r, shift right=0.65ex] & X_{\perf}\times_X X_{\perf} \arrow[r, shift left=1.3ex] \arrow[r, shift right=1.3ex] \arrow[r] &\ldots
\end{tikzcd}
\end{equation}
consists of quasiregular semiperfect schemes (i.e. schemes that locally are spectra of quasiregular semiperfect $k$-algebras), cf. \cite[Remark 8.15]{bms2}. For every $n\geq 0$, denote by $f_n:X_{\perf}^{\times_X (n+1)}\to X$ the structure morphism down to $X$.

By descent for Frobenius pullback of the cotangent complex established in Lemma \ref{stacks: frobenius pullback descent}, Frobenius pullback of the de Rham complex also satisfies flat descent, hence the complex $F_X^*F_{X*}\Omega^{\bullet}_X$ can be described by as the totalization of the following cosimplicial sheaf on $X$:
\begin{equation}
\begin{tikzcd}
f_{0*}F_{X_{\perf}}^*\dR_{X_{\perf}} \arrow[r, shift left=0.65ex] \arrow[r, shift right=0.65ex] & f_{1*}F_{X^{\times_X 2}_{\perf}}^*\dR_{X^{\times_X 2}_{\perf}} \arrow[r, shift left=1.3ex] \arrow[r, shift right=1.3ex] \arrow[r] &\ldots
\end{tikzcd}
\end{equation}
The decomposition constructed in Proposition \ref{qrsp: affine decomposition} gives rise to the splitting of the conjugate filtration on this cosimplicial sheaf, giving the desired decomposition of $F^*_XF_{X*}\Omega^{\bullet}_{X}$.
\end{proof}

\appendix
\section{Automatic continuity of module homomorphisms}\label{continuity}

In this appendix we prove that often a map out of an infinite product of modules is automatically continuous for the product topology, and apply this to show that a splitting of the map $\cO_X\to F_*W(\cO_X)/p$ gives rise to a $n$-quasi-$F$-splitting for some $n\geq 1$. The following result has been shown by Lady \cite[Example 6]{lady}, we give here a self-contained proof. Arguments of this type go back to Specker \cite[Satz III]{specker}.

\begin{lm}\label{continuity: slender}
Let $R$ be a smooth domain over a field $k$, of Krull dimension $>0$. Then
\begin{enumerate}
\item For a collection of $R$-modules $M_1, M_2,\ldots$ any map $\prod\limits_{i=1}^{\infty}M_i\to R$ factors through a finite product $\prod\limits_{i=1}^{N}M_i$.
\item Given an inverse system $\ldots\twoheadrightarrow N_2\twoheadrightarrow N_1$ of projective $R$-modules indexed by $\bN$, with all transition maps surjective, for any finite projective $R$-module $M$ the natural map is an isomorphism:
\begin{equation}
\colim\limits_i\Hom_R(N_i,M)\to \Hom_R(\lim\limits_{i}N_i, M)
\end{equation}
\end{enumerate}
\end{lm}
\begin{proof}
Let us first deduce (2) from (1). Choosing arbitrary sections of the surjections $N_{i+1}\to N_i$ we may assume that the inverse system has the form $N_i=\prod\limits_{j=1}^i P_j$, for some projective $R$-modules $P_1,P_2,\ldots$, with transition maps $\prod\limits_{j=1}^{i+1} P_j\to \prod\limits_{j=1}^i P_j$ given by the projection onto first $j$ factors. Since $M$ can be embedded into a finite free $R$-module, we may assume that $M=R$, so we arrived at a special case of (1).

Let us now prove (1). Using Noether normalization, choose a polynomial subring $k[x_1,\ldots, x_n]\subset R$ such that $R$ is a finite $k[x_1,\ldots, x_n]$-module. We assumed that $R$ is a smooth $k$-algebra, in particular, it is a Cohen-Macaulay ring \cite[00NQ]{stacks}, so by miracle flatness \cite[00R4]{stacks} the ring $R$ is flat over $k[x_1,\ldots, x_n]$. Therefore, $R$ is a finite projective $k[x_1,\ldots, x_n]$-module, so it admits a $k[x_1,\ldots,x_n]$-linear embedding into a finite free $k[x_1,\ldots, x_n]$-module (it is in fact free itself by Quillen-Suslin theorem, but we do not need to use this), and therefore it is enough to prove the lemma in the case $R=k[x_1,\ldots, x_n]$. 

We assume that $R$ is the polynomial ring $k[x_1,\ldots, x_n]$ from now on. Given a map of $R$-modules $f:\prod\limits_{i=1}^{\infty}M_i\to R$ we will first prove that it satisfies the following continuity property. Consider any sequence of elements $a_j\in \prod\limits_{i=j}^{\infty}M_i\subset \prod\limits_{i=1}^{\infty}M_i$, for $j\geq 0$. Given any strictly increasing sequence $n_1,n_2,\ldots$ of non-negative integers we denote by $\sum\limits_{i=1}^{\infty}x_1^{n_i}a_i$ the element of $\prod\limits_{i=1}^{\infty}M_i$ whose projection onto $\prod\limits_{i=1}^{j}M_i$ coincides with that of $\sum\limits_{i=1}^{j}x_1^{n_i}a_i$, for all $j\geq 1$. We will check that the composition $\prod\limits^{\infty}_{i=1}M_i\xrightarrow{f}R\hookrightarrow k[x_2,\ldots, x_n][[x_1]]$ takes $\sum\limits_{i=1}^{\infty}x_1^{n_i}a_i$ to the series $\sum\limits_{i=1}^{\infty}x_1^{n_i}f(a_i)$.

Indeed, for every $m$ the difference
\begin{multline}
f(\sum\limits_{i=1}^{\infty}x_1^{n_i}a_i)-\sum\limits_{i=1}^{\infty}x_1^{n_1}f(a_i)=f(\sum\limits_{i=1}^{m}x_1^{n_i}a_i)-\sum\limits_{i=1}^{m}x_1^{n_1}f(a_i)+ \\ +f(\sum\limits_{i=m+1}^{\infty}x_1^{n_i}a_i)-\sum\limits_{i=m+1}^{\infty}x_1^{n_1}f(a_i)=f(\sum\limits_{i=m+1}^{\infty}x_1^{n_i}a_i)-\sum\limits_{i=m+1}^{\infty}x_1^{n_1}f(a_i)
\end{multline}
is divisible by $x_1^{n_{m+1}}$. We assumed that integers $n_{m+1}$ tend to infinity as $m$ goes to infinity, so this difference must be zero in $k[x_2,\ldots, x_n][[x_1]]$. 

We will now conclude that $f$ factors through the projection onto a finite product $\prod\limits_{i=1}^j M_i$. Assume, on the contrary, that for every $j$ there exists an element $a_{j+1}\in \prod\limits_{i=j+1}^{\infty}M_i\subset \prod\limits_{i=1}^{\infty}M_i$ such that $f(a_{j+1})\neq 0$. As we just established, for any increasing sequence $n_1,n_2,\ldots$ the element $\sum\limits_{i=1}^{\infty} x_1^{n_i}a_i$ is mapped to $\sum\limits_{i=1}^{\infty}x_1^{n_i}a_i\in k[x_2,\ldots,x_n][[x_1]]$ by the composition of $f$ with the embedding $R\hookrightarrow k[x_2,\ldots, x_n][[x_1]]$.

But we will now prove that for varying sequences $n_1,n_2,\ldots$ the elements $\sum\limits_{i=1}^{\infty} x_1^{n_i}\cdot f(a_i)$ span a subspace of $k[[x_1,\ldots, x_n]]$ that has uncountable dimension over $k$. Denote the maximal power of $x_1$ appearing in the polynomial $f(a_i)\in k[x_1,\ldots, x_n]$ by $d_i$. For a sequence $\eps_1,\eps_2,\ldots$ of elements of $\{0,1\}$ consider the series $$F_{\eps_1,\eps_2,\ldots}=\sum\limits_{i=1}^{\infty} x_1^{d_1+\ldots+d_{i-1}+2(i-1)+\eps_i}\cdot f(a_i)$$ For varying sequences $(\eps_i)_{i\in\bN}$ they form a linearly independent set, because the power series $F_{\eps_1,\eps_2,\ldots}$ has a non-zero coefficient of $x_1^{d_1+\ldots+d_i+2i-1}$ if and only if $\eps_i=1$. But, by construction, all $F_{\eps_1,\eps_2,\ldots}$ are in the image of the map $f$ and are therefore contained in the subspace $R\subset k[[x_1,\ldots, x_n]]$ of countable dimension over $k$, which gives us a contradiction.
\end{proof}

\begin{cor}\label{continuity: qf splitting}
Let $X$ be a smooth scheme of finite type over a perfect field $k$. If the map $\cO_X\to F_*W(\cO_X)/p$ admits an $\cO_X$-linear splitting, then so does the map $\cO_X\to F_*W_n(\cO_X)/p$, for some $n\geq 1$.
\end{cor}

\begin{proof}
We can right away discard all zero-dimensional connected components of $X$ because over those the maps in question are isomorphisms. We therefore assume that every connected component of $X$ is positive-dimensional.

For each $n$, the sheaf $F_*W_n(\cO_X)/p$ is a locally free sheaf of $\cO_X$-modules of finite rank, see e.g. \cite[Proposition 2.9]{kttwyy}. Since $X$ is quasi-compact, to prove that a section $\tau:F_*W(\cO_X)/p\to \cO_X$ factors through $F_*W_n(\cO_X)/p$ for some $n$, it suffices to do so Zariski-locally on $X$. On an affine open $\Spec R\subset X$ this follows from Lemma \ref{continuity: slender} applied to the inverse system $N_i=F_*W_i(R)/p$ and $M=R$.
\end{proof}

\bibliographystyle{alpha}
\bibliography{bibfrobenius}

\end{document}